\RequirePackage{fix-cm}
\documentclass[smallextended]{svjour3}       
\smartqed  
\usepackage{amsmath, amssymb, bm, mathrsfs}
\usepackage{cite, color, graphicx}
\usepackage[colorlinks=true,linkcolor=blue,backref=page]{hyperref}
\renewcommand*{\backref}[1]{}
\makeatletter
\def\widebar{\accentset{{\cc@style\underline{\mskip10mu}}}}
\makeatother

\spnewtheorem{theorem}{Theorem}[section]{\bfseries}{\itshape}

\spnewtheorem{lemma}[theorem]{Lemma}{\bfseries}{\itshape}
\spnewtheorem{corollary}[theorem]{Corollary}{\bfseries}{\itshape}
\spnewtheorem{problem}[theorem]{Problem}{\bfseries}{\itshape}
\spnewtheorem{definition}[theorem]{Definition}{\bfseries}{\itshape}
\spnewtheorem{example}[theorem]{Example}{\bfseries}{\itshape}
\spnewtheorem{proposition}[theorem]{Proposition}{\bfseries}{\itshape}

\spnewtheorem{remark}[theorem]{Remark}{\bfseries}{\upshape}
\spnewtheorem{assumption}[theorem]{Assumption}{\bfseries}{\itshape}

\newcommand{\im}{\mathop{\rm Im}\nolimits}
\newcommand{\re}{\mathop{\rm Re}\nolimits}
\newcommand{\esssup}{\mathop{\rm ess~sup}}
\newcommand{\ran}{\mathop{\rm ran}}

 \makeatletter

\@addtoreset{equation}{section}
\makeatother

\begin{document}

\title{Semi-uniform Input-to-state Stability
	of Infinite-dimensional  Systems
	\thanks{This work was supported by JSPS KAKENHI Grant Number JP20K14362.}
}

\author{Masashi Wakaiki   
}

\institute{M.~Wakaiki \at
              Graduate School of System Informatics, Kobe University, Nada, Kobe, Hyogo 657-8501, Japan \\
              Tel.: +8178-803-6232 \\
              Fax: +8178-803-6392\\
              \email{wakaiki@ruby.kobe-u.ac.jp}           
}

\date{Received: date / Accepted: date}

\maketitle

\begin{abstract}
We introduce the notions of semi-uniform input-to-state stability 
and its subclass, polynomial input-to-state stability,
for
infinite-dimensional systems.
We establish a characterization of semi-uniform input-to-state
stability based on attractivity properties as in the uniform case.
Sufficient conditions for linear systems to be polynomially 
input-to-state stable are  provided, which
restrict the range of the input operator depending on
the rate of polynomial decay of the product of the semigroup and
the resolvent of its generator.
We also show that a class of 
bilinear systems are polynomially integral input-to-state
stable under a certain smoothness assumption on nonlinear operators.

\end{abstract}

\section{Introduction}
Consider a semi-linear system with state space $X$ and
input space $U$ (both Banach spaces):
\begin{equation}
\label{eq:system_intro}
\dot x(t) = Ax(t) +F\big(x(t),u(t)\big),\quad t \geq 0;\qquad x(0) = x_0 \in X,
\end{equation}
where $A$ with domain $D(A)$
is the generator of a $C_0$-semigroup $(T(t))_{t\geq 0}$
on $X$
and $F:X\times U \to X$ is a nonlinear operator
such that the solution $x$ of \eqref{eq:system_intro} exists
on $[0,\infty)$ for each essentially bounded input $u$.
We are interested in the cases where 
\begin{itemize}
\item
the $C_0$-semigroup $(T(t))_{t\geq 0}$ is {\em semi-uniformly stable}, that is,
$(T(t))_{t\geq 0}$ is uniformly bounded and 
$\| T(t)(I-A)^{-1} \| \to 0$ as $t\to \infty$;
\item
the $C_0$-semigroup $(T(t))_{t\geq 0}$ is {\em polynomially stable with 
	parameter $\alpha>0$},
that is,
$(T(t))_{t\geq 0}$ is semi-uniformly stable and 
$\|T(t)(I-A)^{-1} \| = O(t^{-1/\alpha})$ as $t\to \infty$, which means that
there is a constant $M>0$ such that for all $t>0$,
\[
\|T(t)(I-A)^{-1} \| \leq \frac{M}{t^{1/\alpha}}.
\]
\end{itemize}
In this paper, we introduce and study new notions of
input-to-state stability (ISS), which are closely
related to semi-uniformly stable semigroups and
polynomially stable semigroups.

The concept of ISS has been introduced
for ordinary differential equations in \cite{Sontag1989}.
This concept combines asymptotic stability with respect to initial states
and robustness against external inputs.
Motivated by robust stability analysis of partial differential equations,
ISS has been recently studied  for 
infinite-dimensional systems, e.g., in \cite{Jayawardhana2008,
	Dashkovskiy2013,
	Karafyllis2016,
	Nabiullin2018,
	Mironchenko2018, Hosfeld2022,
	Schmid2019,
	Jacob2018,Jacob2020}; 
see also the survey \cite{Mironchenko2020}.
Exponential stability of $C_0$-semigroups plays an important role
in the theory of uniform ISS.
The concept of ISS related to
strong stability of $C_0$-semigroups has been also 
introduced in \cite{Mironchenko2018}.
This stability concept is called strong ISS.

Exponential stability of $C_0$-semigroups is a strong property
in terms of quantitative asymptotic character and robustness against 
perturbations.
However, we sometimes encounter systems that is strongly stable but
not exponentially stable.
Strong stability is distinctly qualitative in character unlike exponential
stability and has a much
weaker asymptotic property than exponential stability.
Hence, it does not hold in general that
the system \eqref{eq:system_intro} with generator $A$ of a strongly stable semigroup
is {\em uniformly globally stable},
that is, there exist functions $\gamma,\mu \in \mathcal{K}_{\infty}$ such that 
\[
\|x(t)\| \leq \gamma(\|x_0\|) + \mu 
\left(\esssup_{0\leq t < \infty} \|u(t)\|_U\right)
\] 
for all $x_0 \in X$, essentially bounded functions
$u: [0,\infty)\to U$, and $t\geq 0$,
even when $F(\xi,v) = Bv$ ($\xi \in X$, $v \in U$) for some bounded linear operator 
$B$ from $U$ to $X$, as shown in
Theorem~3 of \cite{Nabiullin2018}.
Here  $\mathcal{K}_\infty$ 
is the set of the classic comparison functions from nonlinear systems theory;
see the notation paragraph at the end of this section.
The same is true for the sets $\mathcal{KL}$ and $\mathcal{K}$ we will use 
below.

Semi-uniform stability and its subclass, polynomial stability, lie
between the two notions of semigroup stability, exponential stability
and strong stability, in the sense that 
semi-uniform stability leads to the quantified asymptotic behavior
of trajectories with initial states
in the domain of the generator.
Semi-uniformly stable semigroups 
have been extensively studied, and it has been shown that 
various partial differential equations such as weakly damped wave equations
are semi-uniformly stable. We refer, for example, to \cite{Liu2005PDR, 
	Batkai2006, Batty2008,
	Borichev2010, Paunonen2011,Paunonen2012SS,Paunonen2013SS,
	Paunonen2014OM,Rozendaal2019,
	Chill2019,Chill2020, Su2020} for the developments of semi-uniformly stable 
semigroups and polynomially stable semigroups.
The main motivation of introducing semi-uniform and polynomial versions of 
ISS
is to bridge the gap between uniform ISS and strong ISS as in the semigroup case.

In a manner analogous to semi-uniform stability of semigroups,
we define semi-uniform ISS of the system \eqref{eq:system_intro}
as follows.
The semi-linear 
system \eqref{eq:system_intro} is {\em semi-uniformly ISS}
if 
the system is uniformly globally stable and if
there exist functions $\kappa \in \mathcal{KL}$ and $\mu \in \mathcal{K}_{\infty}$ 
such that 
\begin{equation}
\label{eq:semi_uniform_ISS_intro}
\|x(t)\| \leq \kappa(\|x_0\|_A,t) + \mu \left(\esssup_{0\leq t < \infty} \|u(t)\|_U\right)
\end{equation}
for all $x_0 \in D(A)$, essentially bounded functions
$u: [0,\infty)\to U$, and $t\geq 0$, where $\|\cdot\|_A$
is the graph norm of $A$, i.e., $\|x_0\|_A := 
\|x_0\| + \|Ax_0\|$ for $x_0 \in D(A)$.
We provide a characterization of semi-uniform ISS
based on attractivity properties called 
the limit property and the asymptotic gain property.
These properties have been introduced in \cite{Sontag1996}
 in order to
characterize ISS of ordinary
differential equations. For infinite-dimensional systems,
such attractivity-based 
characterizations have been established for
uniform ISS \cite[Theorem~5]{Mironchenko2018}, strong ISS \cite[Theorem~12]{Mironchenko2018}, and weak ISS \cite[Theorem~3.1]{Schmid2019}.
Using the attractivity properties, we show that 
semi-uniform ISS implies
strong ISS for linear systems and bilinear systems.

The semi-linear 
system \eqref{eq:system_intro} is called {\em polynomially ISS
with parameter $\alpha >0$} 
if the system is semi-uniformly ISS and if a $\mathcal{KL}$ function $\kappa$ 
in \eqref{eq:semi_uniform_ISS_intro} satisfies
$\kappa(r,t) =
O(t^{-1/\alpha})$ as $t\to \infty$ for all $r > 0$.
We study
polynomial ISS of
 the linear system 
 \begin{equation}
 \label{eq:system_intro_linear}
 \dot x(t) = Ax(t) +Bu(t),\quad t \geq 0;\qquad x(0) = x_0 \in X,
 \end{equation}
 where $A$ is the generator of
a polynomially stable semigroup $(T(t))_{t\geq 0}$ with parameter $\alpha >0$
on $X$
and the input operator 
$B$ is a bounded 
linear operator from $U$ to $X$.
It is readily verified that 
polynomial ISS is equivalent to infinite-time admissibility
of input operators
(together with polynomial stability of $C_0$-semigroups);
see, e.g., \cite{Weiss1989} and \cite[Chapter 4]{Tucsnak2009} for admissibility.
Using this equivalence,
we show that the linear system~\eqref{eq:system_intro_linear} is
polynomial ISS if $B$
maps into the domain of the fractional power of $-A$
with exponent $\beta>\alpha$.
Similar conditions are placed for the range of a
perturbation operator in the robustness analysis of polynomial 
stability \cite{Paunonen2011,Paunonen2012SS,Paunonen2013SS}.
This sufficient condition is 
refined in the case where
$A$ is a diagonalizable operator (see Definition~\ref{def:diagonalizable})
on a Hilbert space
and $B$ is of finite rank. Moreover, when
the eigenvalues $(\lambda_n)_{n \in \mathbb{N}}$ of 
the diagonalizable operator $A$
satisfy $|\im \lambda_n - \im \lambda_m| \geq d$ for some $d >0$
and for all distinct $n,m \in \mathbb{N}$ with $\lambda_n,\lambda_m $ near the imaginary axis,
we give a necessary and sufficient condition 
for polynomial ISS. To this end, we utilize
the relation between Laplace-Carleson embeddings and infinite-time 
admissibility established in Theorem~2.5 of \cite{Jacob2021_arXiv}.

Even in the linear case, uniform global stability and hence polynomial 
ISS impose a strict condition on boundedness of input operators when
$C_0$-semigroups are 
polynomially stable but not exponentially stable.
This observation motivates us to study
a variant of ISS called integral ISS \cite{Sontag1998}. We introduce
a notion of {\em polynomial integral ISS with parameter $\alpha>0$}, 
which means that there exist functions 
$\kappa \in \mathcal{KL}$,
$\gamma, \theta \in \mathcal{K}_{\infty}$, and $\mu \in \mathcal{K}$ such that 
the following three conditions hold: 
(i) $\|x(t)\| \leq \gamma(\|x_0\|)$ for all $x_0 \in X$ and $t \geq 0$
in the zero-input case $u(t)\equiv 0$; (ii) it holds that
\[
\|x(t)\| \leq 
\kappa (\|x_0\|_A, t) + \theta \left(
\int^t_0 \mu (\|u(s)\|_U)\mathrm{d}s
\right)
\]
for all $x_0 \in D(A)$, essentially bounded functions
$u: [0,\infty)\to U$, and $t\geq 0$;
(iii) $\kappa(r,t) = O(t^{-1/\alpha})$ as $t\to \infty$ 
for all $r >0$.
By definition,
we immediately see that the linear system  \eqref{eq:system_intro_linear}
 is polynomially integral ISS
for all generators of polynomially stable semigroups and 
all bounded input operators.
Moreover, 
we prove that 
bilinear systems are also polynomially integral ISS
provided that 
the product  of the $C_0$-semigroup and
the nonlinear operator has the same polynomial decay rate as $\|T(t)(I-A)^{-1}\|$. 
This result is a polynomial analogue 
of Theorem~4.2 in \cite{Mironchenko2016} on
uniform integral ISS.

This paper is organized as follows.
In Section~\ref{sec:basic_facts}, we review some basic facts
on semi-uniform stability and polynomial stability of $C_0$-semigroups.
In Section~\ref{sec:semi_uniform_ISS}, we provide a characterization
of semi-uniform ISS and investigate the relation between
semi-uniform ISS and strong ISS.
Polynomial ISS of linear systems and 
polynomial integral ISS of bilinear systems 
are studied in Sections~\ref{sec:PolyISS} and \ref{sec:PolyiISS},
respectively.

\paragraph*{Notation:}~
Let $\mathbb{N}_0$ and $\mathbb{R}_+$ denote the set of nonnegative integers
and the set of nonnegative real numbers, respectively.
Define 
$i\mathbb{R} := \{i s : s \in \mathbb{R}\}$.
For real-valued functions $f,g$ on $\mathbb{R}$, we write
\[
f(t) = O\big(g(t) \big)\qquad \text{as $t \to \infty$}
\] 
if
there exist $M>0$ and $t_0 \in \mathbb{R}$ such that 
$f(t) \leq Mg(t)$ for all $t \geq t_0$.
Let $X$ and $Y$ be Banach spaces. 
The space of
all bounded linear operators from $X$ to $Y$ is denoted by
$\mathcal{L}(X,Y)$. We write $\mathcal{L}(X) := \mathcal{L}(X,X)$.
The domain and the range of a linear operator $A:X \to Y$ are denoted by $D(A)$ and $\ran(A)$,
respectively.
We denote by $\sigma(A)$ and $\varrho(A)$
the spectrum and 
the resolvent set of a linear operator $A:D(A) \subset X \to X$, respectively.
We write
$R(\lambda,A) := (\lambda I - A)^{-1}$ for $\lambda \in \varrho(A)$.
The graph norm $\|\cdot\|_A$ of a linear operator $A:D(A) \subset X \to X$
is defined by $\|x\|_A := \|x\|+\|Ax\|$ for $x \in D(A)$.
We denote by $L^{\infty}(\mathbb{R}_+,X)$ 
the space of all measurable functions $f : \mathbb{R}_+ \to X$ such that 
$\|f\|_{\infty} := \esssup_{t \in \mathbb{R}_+ }\|f(t)\|< \infty$. 
We denote by $C(\Omega,X)$
the space of all continuous functions from a topological space $\Omega$ to $X$. 
Let  $Z$ and $W$ be Hilbert spaces. 
The Hilbert space adjoint of $T \in \mathcal{L}(Z,W)$ is denoted
by $T^*$.

Classes of
comparison functions for ISS are defined as follows:
\begin{align*}
\mathcal{K} &:= 
\{
\mu \in C(\mathbb{R}_+,\mathbb{R}_+) :
\text{$\mu$ is strictly increasing and $\mu(0) = 0$} \} \\
\mathcal{K}_{\infty} &:= 
\{
\mu \in \mathcal{K} : 
\text{$\mu$ is unbounded}
\} \\
\mathcal{L} &:=
\left\{
\mu \in C(\mathbb{R}_+,\mathbb{R}_+) :
\text{$\mu$ is strictly decreasing with $
	\displaystyle
	\lim_{t\to \infty} \mu(t) = 0$} \right\}\\
\mathcal{KL} &:=
\{
\kappa \in C(\mathbb{R}_+ \times \mathbb{R}_+,\mathbb{R}_+) :
\text{$\kappa(\cdot , t) \in \mathcal{K}~~\forall t \geq 0$ and 
	$\kappa(r,\cdot) \in \mathcal{L}~~\forall r >0$}
\}.
\end{align*}

\section{Basic facts on semi-uniform stability
		and polynomial stability of  semigroups}
\label{sec:basic_facts}
We start by recalling the notion of semi-uniform stability 
of $C_0$-semigroups introduced in Definition~1.2 of \cite{Batty2008}.
\begin{definition}
	{\em
		A $C_0$-semigroup $(T(t))_{t\geq 0}$
		on a Banach space with generator $A$
is
{\em semi-uniformly stable}  if 
$(T(t))_{t\geq 0}$ is uniformly bounded and satisfies
\begin{equation}
\label{eq:T_resolvent_conv}
\|T(t)R(1,A)\| \to 0\qquad \text{as $t \to \infty$}.
\end{equation}
	}
\end{definition}

The following characterization of \eqref{eq:T_resolvent_conv}  in terms of
the intersection of $\sigma(A)$ with $i \mathbb{R}$
has been established in Theorem~1.1 of \cite{Batty2008}.

\begin{theorem}
	\label{thm:equivalence_SUS}
	Let 
	$A$ be the generator of a uniformly bounded semigroup 
	$(T(t))_{t\geq 0}$
	on a Banach space.
	Then \eqref{eq:T_resolvent_conv} holds 
	if and only if
	$\sigma(A) \cap i \mathbb{R}$ is empty.
\end{theorem}

Quantitative statements on the decay rate, as $t\to \infty$, of $\|T(t)R(1,A)\| $
and the blow-up rate, as $s \to \infty$, of $R(is,A)$ have been 
also given, e.g., in \cite{Liu2005PDR,Batkai2006,
	Batty2008,Borichev2010,Rozendaal2019,Chill2019}.
In particular, we are interested in semi-uniform stability
with polynomial decay rates studied in 
\cite{Liu2005PDR, Batkai2006, Borichev2010}.
Note that $\|T(t)R(1,A)\|$  and $\|T(t)A^{-1}\|$ are 
asymptotically of the same order,
which easily follows from the resolvent equation.
\begin{definition}
	{\em
		A $C_0$-semigroup $(T(t))_{t\geq 0}$
		on a Banach space  with generator $A$ is {\em polynomially 
			stable with parameter $\alpha >0$} if $(T(t))_{t\geq 0}$ is semi-uniformly stable and satisfies
		\begin{equation}
		\label{eq:poly_estimate}
		\|T(t)A^{-1}\| = O\left(\frac{1}{t^{1/\alpha}}
		\right)\qquad \text{as $t \to \infty$}.
		\end{equation}
	}
\end{definition}

For
the generator $A$ of a uniformly bounded semigroup,
$-A$ is sectorial in the sense of
\cite[Chapter~2]{Haase2006}. Therefore, if $A$ is injective, then 
the fractional power $(-A)^{\beta}$ is well defined for $\beta \in \mathbb{R}$.
The following result gives  the rate of polynomial decay of $\|T(t)(-A)^{-\beta}\|$ 
for $\beta>0$;
see Proposition~3.1 of \cite{Batkai2006} for the proof.
\begin{proposition}
	\label{prop:decay_rate}
	Let $(T(t))_{t\geq 0}$ be a uniformly bounded semigroup  on
	a Banach space with generator $A$ such that $0 \in  \varrho (A)$.
	For fixed $\alpha ,\beta >0$, 
	\begin{equation*}
		\|T(t)A^{-1}\| = O\left(\frac{1}{t^{1/\alpha}}
		\right)\qquad \text{as $t \to \infty$}
	\end{equation*}
	if and only if
	\[
	\|T(t)(-A)^{-\beta}\| = O\left(\frac{1}{t^{\beta/\alpha}}
	\right)\qquad \text{as $t \to \infty$}.
	\]	
\end{proposition}

For $C_0$-semigroups generated by normal 
operators on Hilbert spaces,
a spectral condition equivalent to
polynomial decay is known.
The proof can be found in Proposition~4.1 of \cite{Batkai2006}.
\begin{proposition}
	\label{prop:spectral_prop}
	Let $(T(t))_{t\geq 0}$ 
	be the $C_0$-semigroup on a Hilbert space
	generated by a normal operator $A$ whose spectrum $\sigma(A)$ 
	is contained 
	in the open left half-plane
	$\{\lambda \in \mathbb{C}:\re \lambda <0 \}$.
	For a fixed $\alpha>0$,
	\begin{equation*}
		\|T(t)A^{-1}\| = O\left(\frac{1}{t^{1/\alpha}}
		\right)\qquad \text{as $t \to \infty$}
	\end{equation*}
	if and only if there exist $C,p>0$ such that 
	\[
	|\im \lambda| \geq \frac{C}{|\re \lambda|^{1/\alpha}}
	\]
	for all $\lambda \in \sigma(A)$ with $\re \lambda > -p$.
\end{proposition}

\section{Characterization of semi-uniform input-to-state stability}
\label{sec:semi_uniform_ISS}
In this section, we first present the nonlinear system we consider and 
introduce the notion of semi-uniform input-to-state stability.
Next, we develop a characterization of this stability.
Finally, the relation between semi-uniform input-to-state stability and 
 strong input-to-state stability 
is investigated.

\subsection{System class}
	Let $X$ and $U$ be Banach spaces with norm $\|\cdot\|$ and
$\|\cdot\|_U$, respectively. Let
$\mathcal{U}$ be a normed vector space contained in the space
$L^1_{\text{loc}}(\mathbb{R}_+,U)$ of
all locally integrable functions from $\mathbb{R}_+$ to $U$. We denote by
$\|\cdot\|_{\mathcal{U}}$ the norm on $\mathcal{U}$.
Assume that
	$u(\cdot + \tau) \in \mathcal{U}$ and 
	$\|u\|_{\mathcal{U}} \geq \|u(\cdot + \tau)\|_{\mathcal{U}} $
	for all $u \in \mathcal{U}$ and $\tau \geq 0$.
We are interested in the case 
$\mathcal{U} = L^{\infty}(\mathbb{R}_+,U)$, but
a general space $\mathcal{U}$ is used
for the definition of semi-uniform input-to-state stability 
and its characterization.

Consider a semi-linear system  with state space $X$
and input space $U$:
\begin{align*}
	\Sigma(A,F) \hspace{10pt}
	\begin{cases}	
		\dot x(t) = Ax(t) + F\big(x(t),u(t)\big),\quad t\geq 0\\ 
		x(0) = x_0,
	\end{cases}
\end{align*}
where 
$A$ is the generator of a $C_0$-semigroup $(T(t))_{t\geq 0}$ on $X$, 
$F:X\times U \to X$ is a nonlinear operator, $x_0 \in X$ is an initial state, and
$u \in \mathcal{U}$ is an input. 

\begin{definition}
	{\em
	Suppose that 
	for every  $\tau>0$, $f \in C([0,\tau],X)$, and $g\in \mathcal{U}$,
	the map $t \mapsto F(f(t),g(t))$ is integrable on $[0,\tau]$.
	For $\tau>0$,
	a function $x \in C([0,\tau],X)$ is called {\em a mild solution of
		$\Sigma(A,F)$
		on $[0,\tau]$} 
	if $x$ satisfies
	the integral equation
	\begin{equation}
	\label{eq:mild_solution}
	x(t) = T(t)x_0 + \int^t_0 T(t-s)  F\big(x(s),u(s)\big) \mathrm{d}s\qquad \forall t \in [0,\tau].
	\end{equation}
	Moreover, we say that $x \in C(\mathbb{R}_+,X)$  is {\em a mild solution of 
		$\Sigma(A,F)$ on $\mathbb{R}_+$} if $x|_{[0,\tau]}$ is a mild solution of 
	$\Sigma(A,F)$ on $[0,\tau]$ for all $\tau >0$.
}
\end{definition}

By Proposition~1.3.4 of \cite{Arendt2001},
the integrability of the map $t \mapsto F(x(t),u(t))$ guarantees that 
the integral in \eqref{eq:mild_solution} exists as 
a Bochner integral and is continuous with respect to $t$.
If the nonlinear operator $F$ satisfies certain Lipschitz conditions, then
$\Sigma(A,F)$ has 
a unique mild solution on $[0,\tau]$ for some $\tau >0$.
We refer, e.g., to \cite{Mironchenko2018, Jacob2020} for 
the mild solution and uniform ISS properties 
of $\Sigma(A,F)$ in the case where $\mathcal{U}$
is the space of piecewise continuous functions that are bounded 
and right-continuous.
In this paper, we sometimes consider a class of 
bilinear systems whose nonlinear operators $F$ satisfy the next
assumption.
\begin{assumption}
	\label{assump:Lipschitz}
	{\em
		The nonlinear operator $F:X\times U \to X$ of 
		$\Sigma(A,F)$ 
		is decomposed into
		$F(\xi,v) = B\xi + G(\xi,v)$ for all $\xi \in X$ and $v \in U$,
		where $B \in \mathcal{L}(X,U)$ and 
		$G:X \times U \to X$ is a nonlinear operator satisfying
		the following conditions:
		\begin{enumerate}
			\item
			$G(0,v) = 0$ for all $v \in U$.
			\item For all $r>0$, there exist $K_r>0$ and $\chi_r \in \mathcal{K}$ such that 
			for all $\xi,\zeta \in X$ with $\|\xi\|,\|\zeta\| \leq r$ and all $v \in U$,
			\[
			\| G(\xi,v) - G(\zeta,v) \| \leq K_r \|\xi - \zeta\| \chi_r (\|v\|_U).
			\]
			\item 
			For all  $\tau>0$ and all functions
			$f \in C ([0,\tau], X)$, $g \in L^{\infty}([0,\tau],U)$,
			the map $t \mapsto G(f(t),g(t))$ is measurable on $[0,\tau]$.
		\end{enumerate}
	}
\end{assumption}

Note that the nonlinear operator $G$ has the following property
under Assumption~\ref{assump:Lipschitz}: For all $r>0$, there exist $K_r>0$ and $\chi_r \in \mathcal{K}$ such that 
for all $\xi\in X$ with $\|\xi\|\leq r$ and all $v \in U$,
\[
\| G(\xi,v)  \| \leq K_r \|\xi\| \chi_r (\|v\|_U).
\]
A standard argument using Gronwall's inequality and
Banach's fixed point theorem shows that
there exists 
a unique mild solution of $\Sigma(A,F)$ for $\mathcal{U} = L^{\infty}(\mathbb{R}_+,U)$
under Assumption~\ref{assump:Lipschitz};
see, e.g., Section~4.3.1 of \cite{Cazenave1998} or
Lemma~2.8 of \cite{Hosfeld2022}.  More precisely,
if Assumption~\ref{assump:Lipschitz} holds, then
one of the following statements is true for
all initial states $x_0 \in X$ and all inputs $u \in L^{\infty}(\mathbb{R}_+,U)$:
\begin{enumerate}
	\item 
	There exists a unique mild solution of $\Sigma(A,F)$ on
	$\mathbb{R}_+$.
	\item
	There exist  $t_{\max}\in(0,\infty)$ 
	and $x \in C([0,t_{\max}),X)$ such that 
	$x|_{[0,\tau]}$ is a unique mild solution of $\Sigma(A,F)$ 
	on $[0,\tau]$ for all $\tau \in(0, t_{\max})$ 
	and \[
	\lim_{t \uparrow t_{\max}}\|x(t)\| = \infty.
	\]
\end{enumerate}

Throughout this section, we consider only forward complete systems;
see also \cite[p.~1284]{Sontag1996} and \cite{Angeli1999}
for forward completeness.
\begin{definition}
	{\em
	The semi-linear system $\Sigma(A,F)$  is 
	{\em forward complete} if 
	there exists a unique mild solution of 
	$\Sigma(A,F)$ on $\mathbb{R}_+$
	for all $x_0 \in X$ and $u \in \mathcal{U}$.
}
\end{definition}

We denote by
$\phi(t,x_0,u)$ the unique mild solution of 
the forward complete semi-linear
system $\Sigma(A,F)$ with initial state $x_0 \in X$ and input $u \in \mathcal{U}$,
i.e.,
\[
\phi(t,x_0,u) = T(t)x_0 + \int^t_0 T(t-s)  F\big(\phi(s,x_0,u),u(s)\big) \mathrm{d}s
\qquad \forall t \geq 0.
\]
The mild solution satisfies the cocycle property
\begin{equation}
\label{eq:cocyle}
\phi(t+\tau,x_0,u) = \phi\big(t, \phi(\tau,x_0,u),u(\cdot + \tau)\big)
\end{equation}
for all $x_0 \in X$, $u \in \mathcal{U}$, and $t,\tau \geq 0$.

\subsection{Definition of semi-uniform input-to-state stability}
For the forward complete semi-linear system $\Sigma(A,F)$, we introduce
the notion of semi-uniform input-to-state stability.
Before doing so, we recall the definition of uniform global stability; see 
\cite[p.~1285]{Sontag1996} and \cite[Definition~6]{Mironchenko2018}.
\begin{definition}
	{\em
	The semi-linear system $\Sigma(A,F)$ is called 
	{\em 
		uniformly globally stable (UGS)} if 
	the following two conditions hold:
	\begin{enumerate}
		\item
		$\Sigma(A,F)$ is forward complete.
		\item
		There exist 
		$\gamma, \mu \in \mathcal{K}_{\infty}$ such that 
		\begin{equation}
		\label{eq:UGS}
		\|\phi(t,x_0,u)\| \leq 
		\gamma (\|x_0\|) + \mu (\|u\|_{\mathcal{U}})
		\end{equation}
		for all $x_0 \in X$, $u \in \mathcal{U}$, and $t \geq 0$.
	\end{enumerate}
}
\end{definition}

\begin{definition}
	\label{def:SUISS}
	{\em
	The semi-linear system $\Sigma(A,F)$ is called 
	{\em semi-uniformly 
		input-to-state stable (semi-uniformly ISS)} if 
	the following two conditions hold:
	\begin{enumerate}
		\item
		$\Sigma(A,F)$ is UGS.
		\item
		There exist 
		$\kappa \in \mathcal{KL}$ and $\mu \in \mathcal{K}_{\infty}$ such that
		\begin{equation}
		\label{eq:semi_ISS}
		\|\phi(t,x_0,u)\| \leq 
		\kappa (\|x_0\|_A, t) + \mu (\|u\|_{\mathcal{U}})
		\end{equation}
		for all $x_0 \in D(A)$, $u \in \mathcal{U}$, and $t \geq 0$.
	\end{enumerate}
	In particular, if there exists $\alpha >0$ such that for
	all $r > 0$,
	$\kappa(r,t) = O(t^{-1/\alpha})$ as $t \to \infty$,  then $\Sigma(A,F)$ is called 
	{\em polynomially  input-to-state stable (polynomially ISS) with parameter $\alpha >0$}.
}
\end{definition}

Assume that the nonlinear operator $F:X\times U \to X$ satisfies 
$F(\xi,0) = 0$
for all $\xi \in X$.
Then one can easily see that 
if 
the semi-linear system $\Sigma(A,F)$ is 
semi-uniformly (resp. polynomially)
ISS, then $A$ 
generates a semi-uniformly (resp. polynomially) 
stable semigroup $(T(t))_{t\geq0}$ on $X$.
In fact,
$\phi(t,x_0,0) = T(t)x_0$ for all $x_0 \in X$ and $t \geq 0$ by assumption.
Therefore,
$(T(t))_{t\geq 0}$ is uniformly bounded by UGS with $u(t) \equiv 0$.
Take $\xi \in X$ with $\|\xi\| = 1$.
Then
\[
\|AR(1,A)\xi\| \leq 1+\|R(1,A)\|.
\]
Since the inequality  \eqref{eq:semi_ISS} with $u(t) \equiv 0$ yields
\begin{align*}
\|T(t)R(1,A)\xi\| 
&\leq \kappa(\|R(1,A)\xi\|+ \|AR(1,A)\xi\|,t)   \\
&\leq \kappa(1+2\|R(1,A)\|,t) 
\end{align*}
for all $t \geq 0$,
it follows that 
\[
\|T(t)R(1,A)\| \leq \kappa(1+2\|R(1,A)\|,t) \to 0\qquad \text{as $t \to \infty$}.
\]
Hence $(T(t))_{t\geq 0}$ is semi-uniformly stable.
Note that 
semi-uniform stability of $(T(t))_{t\geq 0}$ 
generated by $A$ is equivalent to
$i \mathbb{R} \subset \varrho(A)$ by 
Theorem~\ref{thm:equivalence_SUS}.
A similar calculation shows that
$
\|T(t)A^{-1}\| \leq  \kappa(1+\|A^{-1}\|,t)
$
for all $t \geq 0$. Thus, 
polynomial ISS 
of $\Sigma(A,F)$
implies polynomial stability of $(T(t))_{t\geq 0}$.

We conclude this subsection by giving an example of 
polynomially ISS (but not necessarily uniform ISS)
nonlinear systems.
\begin{example}
	{\em
	Let $X$ and $U$ be Banach spaces.
	Let $A$ be the generator of a polynomially
	stable semigroup $(T(t))_{t \geq 0}$ with parameter $\alpha >0$ on $X$ 
	and let
	$H \in \mathcal{L}(U,X)$ satisfy
	$\ran(H) \subset D((-A)^\beta)$ for some $\beta > \alpha$.
	Assume that
	 $q : \mathbb{R}_+ \to  \mathbb{R}$ satisfies the following conditions:
	\begin{enumerate}
		\item $q(0) = 0$.
		\item For all $r >0$, there exists $K_r >0$ such that 
		\[
		|q(z) - q(w)| \leq K_r |z-w|\qquad \forall z,w \in [0,r].
		\]
		\item
		$\sup_{z \geq 0} |q(z)| < \infty$.
	\end{enumerate}
	Define a 
	nonlinear operator $F:X \times U \to X$ by
	\[
	F(\xi,v) := q (\|\xi\|) Hv,\quad \xi \in X,~v \in U.
	\]
	A routine calculation shows that Assumption~\ref{assump:Lipschitz}
	holds for the nonlinear operator $F$.
	
	We show that $\Sigma(A,F)$ is polynomially ISS with parameter $\alpha$
	for $\mathcal{U}= L^{\infty}(\mathbb{R}_+,U)$.
	By Proposition~\ref{prop:decay_rate}, there is a constant $M>0$ such that 
	\begin{equation*}
		\|T(t)(-A)^{-\beta}\| \leq \frac{M}{(t+1)^{\beta/\alpha}}\qquad 
		\forall t \geq 0.
	\end{equation*}
	Since $(-A)^{\beta}$ is closed and since
	$\ran(H) \subset D((-A)^\beta)$, it follows that
	 $(-A)^{\beta} H \in \mathcal{L}(U,X)$.
	 Let $c :=\sup_{z \geq 0} |q(z)| < \infty$ and $t >0$.
	We obtain
	\begin{align*}
	\left\|
	\int^t_0 T(t-s) F\big(x(s),u(s) \big) \mathrm{d}s 
	\right\| 
	&=
	\left\|
	\int^t_0 T(t-s) (-A)^{-\beta} (-A)^\beta F\big(x(s),u(s) \big) \mathrm{d}s
	\right\| \\
	&\leq \int^t_0 \frac{cM \|(-A)^\beta H\| }{(t-s+1)^{\beta/\alpha}} 
	\|u\|_{\infty} \mathrm{d}s \\
	&\leq \frac{\alpha cM \|(-A)^\beta H\|}{\beta -\alpha} \|u\|_{\infty}
	\end{align*}
	for all $x\in C([0,t],X)$ and
	$u \in L^{\infty}(\mathbb{R}_+,U)$.
	From this estimate, we see that $\Sigma(A,F)$ is polynomially ISS with parameter $\alpha$
	for $\mathcal{U}= L^{\infty}(\mathbb{R}_+,U)$.
}
\end{example}

\subsection{Characterization of semi-uniform input-to-state stability}
We define a semi-uniform version of
the properties of uniform attractivity and strong 
attractivity studied in \cite{Mironchenko2018}. The
attractivity properties
has been originally introduced in \cite[pp.~1284--1285]{Sontag1996}
in order to characterize ISS of ordinary differential equations.
\begin{definition}
	{\em 
	The forward complete 
	semi-linear system $\Sigma(A,F)$ has the
	{\em semi-uniform 
		limit property} if 
	there exists 
	$\mu \in \mathcal{K}_{\infty}$ such that 
	the following statement holds:
	For all $\varepsilon, r >0$, there is
	$\tau = \tau(\varepsilon,r) < \infty$ such that for all $x_0 \in D(A)$,
	\begin{equation*}
	\|x_0\|_A \leq r ~~\wedge~~ u \in \mathcal{U}
	\quad \Rightarrow \quad 
	\exists t \leq \tau :
	\|\phi(t,x_0,u)\| \leq 
	\varepsilon + \mu (\|u\|_{\mathcal{U}}).
	\end{equation*}
}
\end{definition}
\begin{definition}
	{\em
	The forward complete 
	semi-linear system $\Sigma(A,F)$ has the
	{\em semi-uniform 
		asymptotic gain property} if 
	there exists  $\mu \in \mathcal{K}_{\infty}$ such that 
	the following statement holds:
	For all $\varepsilon,r >0$, 
	there is $\tau = \tau(\varepsilon,r) < \infty$ such that for all
	$x_0 \in D(A)$ with $\|x_0\|_A \leq r$
	and all $u \in \mathcal{U}$,
	\[
	t \geq \tau \quad \Rightarrow \quad
	\|\phi(t,x_0,u)\| \leq \varepsilon + \mu(\|u\|_{\mathcal{U}}).
	\]
}
\end{definition}

By definition, the asymptotic gain property is stronger than
the limit property. We will see that 
both  properties are equivalent if the system is UGS. Moreover,
based on
these attractivity properties, 
a characterization of semi-uniform ISS is established.
The attractivity-based characterization of ISS is useful when
the construction of a $\mathcal{KL}$ function $\kappa$ is involved. 
The proof for 
the semi-uniform case 
is obtained by a slight modification of
the proof of Theorem~5 in \cite{Mironchenko2018}
 for the uniform case.
We sketch it for the sake of completeness.
\begin{theorem}
	\label{thm:UGS_PLP}
	The following statements on the semi-linear system $\Sigma(A,F)$ 
	are equivalent:
	\begin{enumerate}
		\renewcommand{\labelenumi}{{\em \arabic{enumi}.}}
		\item
		$\Sigma(A,F)$ is semi-uniformly ISS.
		\item
		$\Sigma(A,F)$  is 
		UGS and has the semi-uniform limit property.
		\item
		$\Sigma(A,F)$ is
		UGS and has the semi-uniform 
		asymptotic gain property.
	\end{enumerate}
\end{theorem}

\begin{proof} \mbox{}
	[1. $\Rightarrow $ 2.]
	Suppose that $\Sigma(A,F)$
	is semi-uniformly ISS.
	By definition, $\Sigma(A,F)$ is
	UGS.
	There exist $\kappa \in \mathcal{KL}$ 
	and $\mu \in \mathcal{K}_{\infty}$ such that 
	\[
	\|\phi(t,x_0,t)\| \leq \kappa(\|x_0\|_A,t) + \mu(\|u\|_{\mathcal{U}})
	\]
	for all $x_0 \in D(A)$, $u \in \mathcal{U}$, and $t \geq 0$.
		Take $\varepsilon,r >0$.
	We obtain
	$\kappa(r,\tau) \leq \varepsilon$ for some
	$\tau = \tau(\varepsilon ,r) < \infty$. Therefore,
	if $x_0 \in D(A)$ satisfies $\|x_0\|_A \leq r$, then
	\[
	\|\phi(\tau ,x_0,t)\| \leq \varepsilon + \mu(\|u\|_{\mathcal{U}})\qquad
	\forall u \in  \mathcal{U}.
	\]
	Thus, $\Sigma(A,F)$
	has the semi-uniform limit property.
	
	[2. $\Rightarrow $ 3.] 
	Suppose that $\Sigma(A,F)$ is 
	UGS and has the semi-uniform 
	limit property.
	By assumption, there exist 
	$\gamma, \mu \in \mathcal{K}_{\infty}$ such that the following statement 
	holds:
	For every $\varepsilon,r >0$, 
	there is $\tau = \tau(\varepsilon,r) < \infty$ such that for all
	$x_0 \in D(A)$,
	\begin{align*}
	 \|x_0\|_A \leq r ~~\wedge~~ u \in \mathcal{U}
	 \quad \Rightarrow \quad 
	 \exists t_1 \leq \tau:
	\|\phi(t_1,x_0,u)\| \leq \varepsilon + \mu(\|u\|_{\mathcal{U}}),
	\end{align*}
	and for all $s \geq 0$,
	\begin{align*}
	\big\|\phi\big(s, \phi(t_1,x_0,u),u(\cdot+t_1)\big)\big\| \leq \gamma(\|\phi(t_1,x_0,u)\|)  +\mu(\|u(\cdot+t_1)\|_{\mathcal{U}}).
	\end{align*}
	Using the cocycle property \eqref{eq:cocyle}, we obtain
	\begin{align*}
		\|\phi(t_1+s,x_0,u)\| &= 
		\big\|\phi\big(s, \phi(t_1,x_0,u),u(\cdot+t_1) \big) \big\| \\
		&\leq \gamma \big(\varepsilon + \mu(\|u\|_{\mathcal{U}}) \big) 
		+ \mu(\|u(\cdot+t_1)\|_{\mathcal{U}}).
	\end{align*}
	Since $\gamma(a+b) \leq \gamma(2a) + \gamma(2b)$ for all $a,b \geq 0$,
	it follows that
	\begin{align*}
	\|\phi(t_1+s,x_0,u)\| &\leq 
	\gamma(2\varepsilon) + \gamma\big(2\mu(\|u\|_{\mathcal{U}}) \big) + \mu(\|u\|_{\mathcal{U}}) \\
	&\leq 
	\gamma(2\varepsilon) +  \widetilde \mu(\|u\|_{\mathcal{U}}),
	\end{align*}
	where $\widetilde \mu := 
	\gamma \circ (2\mu) +\mu \in \mathcal{K}_{\infty}$. 
	
	Choose
	$\widetilde \varepsilon, \widetilde r>0$ arbitrarily and set
	\[
	\varepsilon := \frac{\gamma^{-1}(\widetilde \varepsilon)}{2},\quad
	r:= \widetilde r.
	\]
	We have shown that
	there is
	\[
	\tau = \tau(\varepsilon,r) = \tau\left(\frac{\gamma^{-1}
		(\widetilde \varepsilon)}{2},\widetilde r \right)
	< \infty
	\]
	such that for all $x_0 \in D(A)$ with 
	$\|x_0\|_A \leq \widetilde r$ and all $u \in \mathcal{U}$,
	\[
	t \geq \tau \quad \Rightarrow \quad
	\|\phi (t,x_0,u) \| \leq \widetilde \varepsilon + \widetilde \mu (\|u\|_\mathcal{U}).
	\]
	Hence $\Sigma(A,F)$  has the semi-uniform 
	asymptotic gain property.

	[3. $\Rightarrow $ 1.]
	Suppose that $\Sigma(A,F)$ is 
	UGS and has the semi-uniform 
	asymptotic gain property.
	There exist $\gamma,\mu \in \mathcal{K}_{\infty}$ such that
	the following two properties hold:
	\begin{enumerate}
		\renewcommand{\labelenumi}{(\alph{enumi})}
		\item
		For all $x_0 \in X$, $u \in \mathcal{U}$, and $t \geq 0$,
		\begin{equation}
		\label{eq:UGS_proof}
		\|\phi(t,x_0,u)\| \leq 
		\gamma (\|x_0\|) + \mu (\|u\|_\mathcal{U}).
		\end{equation}
		
		\item
		For all $\varepsilon,r >0$, 
		there is $\tau = \tau(\varepsilon,r) < \infty$ such that 
		for all $x_0 \in D(A)$ with $\|x_0\|_A \leq r$
		and all $u  \in \mathcal{U}$,
		\begin{equation}
		\label{eq:SUAG}
		t \geq \tau \quad \Rightarrow \quad
		\|\phi(t,x_0,u)\| \leq \varepsilon + \mu(\|u\|_\mathcal{U}).
		\end{equation}
	\end{enumerate}

	Let $r >0$. Set
	$\varepsilon_n := 2^{-n}\gamma(r)$ for $n \in \mathbb{N}_0$
	and $\tau_0 := 0$.
	By the property \eqref{eq:SUAG},
	there exist
	$\tau_n = \tau_n(\varepsilon_n,r)$, $n \in \mathbb{N}$, such that 
	for all $x_0 \in D(A)$ with $\|x_0\|_A \leq r$
	and all $u  \in \mathcal{U}$,
	\begin{equation}
	\label{eq:taun_bound}
	t \geq \tau_n \quad \Rightarrow \quad
	\|\phi(t,x_0,u)\| \leq \varepsilon_n + \mu(\|u\|_\mathcal{U}).
	\end{equation}
	For $n=0$, we also obtain \eqref{eq:taun_bound} 
	by the property \eqref{eq:UGS_proof} and
	the inequality
	\[
	\gamma(\|x_0\|) \leq \gamma(\|x_0\|_A) \leq \gamma(r) = \varepsilon_0.
	\]
	We may assume
	without loss of generality that 
	$
	\inf_{n \in \mathbb{N}} (\tau_n - \tau_{n-1}) >0.
	$ 
	For  these 
	sequences $(\varepsilon_n)_{n \in \mathbb{N}_0}$ 
	and $(\tau_n)_{n \in \mathbb{N}_0}$, one can construct
	a function $\kappa \in \mathcal{KL}$ satisfying
	\begin{equation}
	\label{eq:eps_bound}
	\varepsilon_n 
	\leq \kappa(r,t) \qquad 
	\forall t \in [\tau_n,\tau_{n+1}),~\forall n \in \mathbb{N}_0;
	\end{equation}
	see	the proof of Lemma~7 of \cite{Mironchenko2018} for the detailed 
	construction. From \eqref{eq:taun_bound} and \eqref{eq:eps_bound},
	we have that for all $x_0 \in D(A)$ with $\|x_0\|_A \leq r$
	and all $u  \in \mathcal{U}$,
	\begin{equation}
	\label{eq:phi_beta_bound}
	\|\phi(t,x_0,u)\| \leq \kappa(r,t)+ \mu(\|u\|_\mathcal{U})\qquad \forall t \geq 0.
	\end{equation}

	Take $x_0 \in D(A)$ and  
	$u \in \mathcal{U}$ arbitrarily.
	If $\|x_0\|_A = 0$, then the property \eqref{eq:UGS_proof} yields
	\[
	\|\phi(t,x_0,u)\| \leq  \mu (\|u\|_\mathcal{U}) =
	\kappa(\|x_0\|_A,t)+ \mu(\|u\|_\mathcal{U})
	\qquad \forall t \geq 0.
	\]
	If $\|x_0\|_A >0$, then 
	it follows from
	\eqref{eq:phi_beta_bound} with $r:= \|x_0\|_A$ that
	\[
	\|\phi(t,x_0,u)\| \leq \kappa(\|x_0\|_A,t)+ \mu(\|u\|_\mathcal{U}) \qquad \forall t \geq 0.
	\]
	Thus, $\Sigma(A,F)$ is semi-uniformly
	ISS.
	\qed
\end{proof}

\subsection{Relation between semi-uniform input-to-state stability and 
	strong input-to-stability}

After recalling the notion of strong input-to-state stability introduced in \cite[Definition~13]{Mironchenko2018},
we study its relation to semi-uniform ISS
with the help of the characterization in Theorem~\ref{thm:UGS_PLP}.

\begin{definition}
	{\em
	The semi-linear system $\Sigma(A,F)$ is {\em strongly input-to-state 
		stable (strongly ISS)} if 
	$\Sigma(A,F)$ is forward complete and if
	there exist $\gamma, \mu \in \mathcal{K}_{\infty}$
	and $\kappa : X \times \mathbb{R}_+ \to \mathbb{R}_+$ such that 
	the following three conditions hold:
	\begin{enumerate}
		\item $\kappa(x_0,\cdot) \in \mathcal{L}$ for all $x_0 \in X$ with $x_0\not=0$.
		\item $\kappa(x_0,t) \leq \gamma(\|x_0\|)$ for all $x_0 \in X$ and $t \geq 0$.
		\item $\|\phi(t,x_0,u)\| \leq \kappa(x_0,t) + \mu (\|u\|_{\mathcal{U}})$
		for all $x_0 \in X$, $u \in \mathcal{U}$, and $t \geq 0$.
	\end{enumerate}
}
\end{definition}

For some special classes of semi-linear systems,
semi-uniform ISS implies strong ISS.
\begin{theorem}
	\label{thm:semi_to_strong}
	Assume that the operator $F$ of $\Sigma(A,F)$ satisfies one of the following 
	conditions: 
	\begin{enumerate}
	\renewcommand{\labelenumi}{{\em \arabic{enumi}.}}
	\item
	There exists $B \in \mathcal{L}(U,X)$ such that 
	$F(\xi,v) = Bv$ for all $\xi \in X$ and $v \in U$.
	\item 
	$F(\xi-\zeta,v) = F(\xi,v) - F(\zeta,v) $ for all  $\xi,\zeta \in X$ and $v \in U$.
	\end{enumerate}
	Then
	semi-uniform ISS implies
	 strong ISS
	for $\Sigma(A,F)$.
\end{theorem}
\begin{proof}
		By Theorem~12 of \cite{Mironchenko2018},
	the semi-linear system $\Sigma(A,F)$ is strongly ISS
	if and only if $\Sigma(A,F)$
	is UGS and has the strong asymptotic gain property, which means that 
	there exists $\mu \in \mathcal{K}_{\infty}$ such that 
	the following statement holds:
	For all $\varepsilon>0$
	and $x_0 \in X$, there exists $\tau = \tau(\varepsilon,x_0)<\infty$
	such that for all $u \in \mathcal{U}$,
	\begin{equation}
	\label{eq:SAGP}
	 t \geq \tau\quad \Rightarrow \quad 
	\|\phi(t,x_0,u)\| \leq 
	\varepsilon + \mu (\|u\|_{\mathcal{U}}).
	\end{equation}
	It suffices to show that 
	semi-uniform
	ISS implies
	the strong asymptotic gain property.
	
	1. Assume that there exists $B \in \mathcal{L}(U,X)$ such that 
	$F(\xi,v) = Bv$ for all $\xi \in X$ and $v \in U$.
	By linearity, $
	\phi(t,x_0,u) = \phi(t,x_0,0) + \phi(t,0,u) 
	$ for all 
	$x_0 \in X$, $u \in \mathcal{U}$, and $t \geq 0$.
	Since $\Sigma(A,F)$ is semi-uniformly ISS, 
	$(T(t))_{t\geq 0}$ is uniformly bounded and 
	$\lim_{t\to \infty}T(t)x_0 = 0$ as $t \to \infty$ for all $x_0 \in D(A)$.
	Hence
	strong stability of $(T(t))_{t\geq 0}$ follows by the density of 
	$D(A)$; see also Proposition~A.3 of \cite{Engel2000}. 
	For all $\varepsilon >0$ and $x_0 \in X$,
	there exists $\tau = \tau(\varepsilon,x_0)<\infty$ such that 
	\[
	\|\phi(t,x_0,0)\| = \|T(t)x_0\| \leq \varepsilon\qquad \forall t \geq \tau.
	\]
	Since $\Sigma(A,F)$ is UGS, there exists 
	$\mu \in \mathcal{K}_{\infty}$ such that 
	\[
	\|\phi(t,0,u)\| \leq  \mu (\|u\|_{\mathcal{U}})
	\qquad \forall u \in \mathcal{U},~\forall t \geq 0.
	\]
	Thus, $\Sigma(A,F)$ has the strong asymptotic gain property.

	2. 	Assume that  
	$F(\xi-\zeta,v) = F(\xi,v) - F(\zeta,v) $ for all  $\xi,\zeta \in X$ and $v \in U$.
	Since $\Sigma(A,F)$  is 
	UGS and has the semi-uniform asymptotic gain property by Theorem~\ref{thm:UGS_PLP},
	there exist $\gamma,\mu \in \mathcal{K}_{\infty}$ such that 
	the following two properties hold:
	\begin{enumerate}
		\renewcommand{\labelenumi}{(\alph{enumi})}
		\item
		For all $x_0 \in X$, $u \in \mathcal{U}$, and $t \geq 0$,
		\begin{equation}
		\label{eq:UGS_proof2}
		\|\phi(t,x_0,u)\| \leq 
		\gamma (\|x_0\|) + \mu (\|u\|_\mathcal{U}).
		\end{equation}
		
		\item
		For all $\varepsilon, r >0$, there exists
		$\tau = \tau(\varepsilon,r) < \infty$ such that for all $x_0 \in D(A)$
		with $\|x_0\|_A \leq r$ and all $u \in \mathcal{U}$,
		\begin{equation}
		\label{eq:SULIM_proof2}
		t \geq \tau\quad \Rightarrow \quad 
		\|\phi(t,x_0,u)\| \leq 
		\varepsilon + \mu (\|u\|_{\mathcal{U}}).
		\end{equation}
	\end{enumerate}

	Take $\varepsilon>0$ and $x_0 \in X$.
	There exists $y_0 \in D(A)$ such that $\|x_0 - y_0\| \leq \gamma^{-1}(\varepsilon/2)$.
	By assumption, for all $u \in \mathcal{U}$,
	$\phi(t,x_0,u)-\phi(t,y_0,u)$ is the mild solution of $\Sigma(A,F)$ with initial state $x_0 -y_0$
	and input $u$. Therefore, the property \eqref{eq:UGS_proof2} implies that 
	\begin{align}
		\|\phi(t,x_0,u)- \phi(t,y_0,u)\| 
		&= \|\phi(t,x_0 - y_0,u)\| \notag \\
		&\leq \gamma(\|x_0 - y_0\|) + \mu(\|u\|_{\mathcal{U}}) \notag \\
		&=\frac{\varepsilon}{2} + \mu(\|u\|_{\mathcal{U}})
		\label{eq:phi_diff}
	\end{align}
	for  all $u \in \mathcal{U}$ and $t \geq 0$.
	
	Since $y_0 \in D(A)$, 
	it follows from the property \eqref{eq:SULIM_proof2} that
	in the case $y_0 \not=0$,
	there exists $\tau = \tau(\varepsilon,\|y_0\|_A) <\infty$ such that 
	for all $u \in \mathcal{U}$,
	\begin{equation}
	\label{eq:y_0_conv}
	 t \geq \tau  \quad \Rightarrow \quad 
	\|\phi(t,y_0,u)\| \leq \frac{\varepsilon}{2} + \mu(\|u\|_{\mathcal{U}}).
	\end{equation}
	In the case $y_0 = 0$, 
	the property \eqref{eq:UGS_proof2} yields that \eqref{eq:y_0_conv} holds 
	with $\tau = 0$.
	Combining the estimates \eqref{eq:phi_diff} and \eqref{eq:y_0_conv},
	we obtain
	\[
	t \geq \tau  \quad \Rightarrow \quad 
	\|\phi(t,x_0,u)\| \leq \varepsilon + \widetilde \mu (\|u\|_{\mathcal{U}})
	\]
	for all $u \in \mathcal{U}$, where $\widetilde \mu  := 2\mu \in \mathcal{K}_{\infty}$.
	Since $y_0$ depends only on $\varepsilon$ and $x_0$, it follows that 
	$\varepsilon$ and $x_0$ determine $\tau = \tau(\varepsilon,\|y_0\|_A) $.
	Thus, $\Sigma(A,F)$ has the strong 
	asymptotic gain property.
	\qed
\end{proof}

Suppose that $\Sigma(A,F)$ is strong ISS. 
If the input $u \in \mathcal{U}$ satisfies 
\[
\lim_{\tau \to \infty}\|u(\cdot + \tau)\|_{\mathcal{U}} \to 0,
\]
then $\|\phi(t,x_0,u)\| \to 0$ as $t \to \infty$ for all $x_0 \in X$;
see Lemma~2.5 of \cite{Schmid2019}, where this convergence result
has been proved under a weaker assumption. We obtain 
a convergence property of semi-uniform ISS
as a corollary
of Theorem~\ref{thm:semi_to_strong}.
\begin{corollary}
	Under the same assumption on the operator $F$ as in Theorem~\ref{thm:semi_to_strong},
	if  $\Sigma(A,F)$ is semi-uniformly
	ISS, then $\|\phi(t,x_0,u)\| \to 0$ as $t \to \infty$
	for all $x_0 \in X$ and all  $u \in \mathcal{U}$ satisfying 
	$\|u(\cdot+\tau)\|_{\mathcal{U}} \to 0$ as $\tau \to \infty$.
\end{corollary}

\section{Polynomial input-to-state stability of linear systems}
\label{sec:PolyISS}
In this section, 
we focus on polynomial ISS of linear systems for $\mathcal{U} = 
L^{\infty}(\mathbb{R}_+,U)$.
First,
we give a sufficient condition for general linear systems to be polynomially ISS.
Next, we consider linear systems with diagonalizable generators and
finite-rank input operators and refine the sufficient condition.
Finally, a necessary and sufficient condition for polynomial ISS
is presented in the case where the eigenvalues of
the diagonalizable generator near
the imaginary axis have uniformly separated imaginary parts.

\subsection{Polynomial input-to-state stability for general linear systems}
Let $X$ and $U$ be Banach spaces.
Consider a linear system  with state space $X$
and input space $U$:
\begin{align*}
	\Sigma_{\mathrm{lin}}(A,B) \hspace{10pt}
	\begin{cases}	
		\dot x(t) = Ax(t) + Bu(t),\quad t\geq 0\\ 
		x(0) = x_0,
	\end{cases}
\end{align*}
where 
$A$ is the generator of a $C_0$-semigroup $(T(t))_{t\geq 0}$ on $X$, 
$B \in \mathcal{L}(U,X)$ is an input operator, $x_0 \in X$ is an initial state, and
$u \in L^{\infty}(\mathbb{R}_+,U)$ is an input. 

To study ISS of linear systems, we employ
the notion of admissibility  studied in
the seminal work \cite{Weiss1989}.
\begin{definition}
	\label{def:admissible}
	{\em
	We call the operator $B \in \mathcal{L}(U,X)$ 
	{\em infinite-time $L^{\infty}$-admissible} for a $C_0$-semigroup 
	$(T(t))_{t\geq 0}$ on $X$
	if there exists a constant $c >0$ such that 
	\[
	\left\|
	\int^t_0 T(s)Bu(s) \mathrm{d}s
	\right\| \leq c \|u\|_{\infty}
	\]
	for all $u \in L^{\infty}(\mathbb{R}_+,U)$ and $t \geq 0$.
}
\end{definition}
Let $(T(t))_{t\geq 0}$ be a $C_0$-semigroup on $X$ and 
let $B \in \mathcal{L}(U,X)$. If there exists
$\mu \in \mathcal{K}_{\infty}$ such that
\[
\left\|
\int^t_0 T(t-s)Bu(s)\mathrm{d}s
\right\| \leq \mu(\|u\|_{\infty})\qquad \forall u \in L^{\infty}(\mathbb{R}_+,U),~
\forall t \geq 0,
\]
then 
\begin{align*}
	\left\|
	\int^t_0 T(s)Bu(s)\mathrm{d}s
	\right\| &=
	\left\|
	\int^t_0 T(s)B \frac{u(s)}{\|u\|_{\infty}}\mathrm{d}s 
	\right\| \|u\|_{\infty}  \leq \mu (1) \|u\|_{\infty}
\end{align*}
for  all $u \in L^{\infty}(\mathbb{R}_+,U) \setminus \{0\}$ and 
$t \geq 0$.
Hence $B$ is infinite-time $L^{\infty}$-admissible for  $(T(t))_{t\geq 0}$.

As in the case of strong ISS \cite[Proposition~1]{Nabiullin2018},
polynomial ISS for $\mathcal{U} = 
L^{\infty}(\mathbb{R}_+,U)$ 
is equivalent to the combination of 
polynomial stability of $C_0$-semigroups and 
infinite-time $L^{\infty}$-admissibility of input operators.
\begin{lemma}
	\label{lem:poly_ISS_admissible}
	Let $X$ and $U$ be Banach spaces. 
	The linear system $\Sigma_{\mathrm{lin}}(A,B)$ 
	is polynomially ISS with parameter $\alpha >0$  for $\mathcal{U}= L^{\infty}(\mathbb{R}_+,U)$ if and only if the $C_0$-semigroup
	$(T(t))_{t\geq 0}$ on $X$
	generated by $A$
	is polynomially stable with parameter $\alpha$ and the input operator
	$B\in \mathcal{L}(U,X)$ is infinite-time $L^{\infty}$-admissible for $(T(t))_{t\geq 0}$.
\end{lemma}
\begin{proof}
	By the remarks following Definitions~\ref{def:SUISS} and 
	\ref{def:admissible}, polynomial ISS of 
	$\Sigma_{\mathrm{lin}}(A,B)$ for $\mathcal{U} = 
	L^{\infty}(\mathbb{R}_+,U)$ implies 
	polynomial stability of $(T(t))_{t\geq 0}$ and infinite-time $L^{\infty}$-admissibility
	of $B$.
	The converse implication immediately follows,
	since there exist constants $M,c>0$ such that 
	\[
	\|\phi(t,x_0,u)\| 
	\leq \frac{M\|x_0\|_A}{(t+1)^{1/\alpha}} + c\|u\|_{\infty}
	\]
	for all $x_0 \in D(A)$, $u \in L^{\infty}(\mathbb{R}_+,U)$, and $t \geq 0$.
	\qed
\end{proof}

We provide a simple sufficient condition for $\Sigma_{\rm lin}(A,B)$
to be polynomially ISS,
by restricting the range of the input operator $B$.
\begin{proposition}
	\label{prop:polyiss}
	Let $X$ and $U$ be Banach spaces. Suppose that 
	$A$ is the generator of a polynomially stable semigroup with parameter $\alpha >0$
	on $X$.
	If $B \in \mathcal{L}(U,X)$ satisfies
	$\ran(B) \subset D((-A)^{\beta})$ for some $\beta > \alpha$,
	then $\Sigma_{\mathrm{lin}}(A,B)$ is polynomially ISS with parameter $\alpha$
	 for $\mathcal{U}= L^{\infty}(\mathbb{R}_+,U)$.
\end{proposition}

\begin{proof}
	Let $(T(t))_{t\geq 0}$ be the polynomially stable semigroup on $X$ 
	generated by $A$.
By Proposition~\ref{prop:decay_rate}, there exists $M>0$ such that 
\begin{equation*}
	\|T(t)(-A)^{-\beta}\| \leq \frac{M}{(t+1)^{\beta/\alpha}}\qquad 
	\forall t \geq 0.
\end{equation*}
	Since $(-A)^{\beta}$ is closed, we have that $(-A)^{\beta} B \in \mathcal{L}(U,X)$ by assumption.
	For all $u \in L^{\infty}(\mathbb{R}_+,U)$ and $t \geq 0$, 
	we obtain
	\begin{align*}
		\left\|
		\int^t_0 T(s) Bu(s) \mathrm{d}s
		\right\| 
		&=
		\left\|
		\int^t_0 T(s) (-A)^{-\beta} (-A)^{\beta}  Bu(s) \mathrm{d}s
		\right\| \\
		&\leq 
		\int^{t}_{0} \frac{M\|(-A)^{\beta}B\|}{(s+1)^{\beta/\alpha}} ~\!
		\|u\|_{\infty} \mathrm{d}s  \\
		&\leq 
		 \frac{\alpha M\|(-A)^{\beta}B\|  }{\beta-\alpha} \|u\|_{\infty}.
 	\end{align*}
	Hence $B$ is infinite-time $L^{\infty}$-admissible for $(T(t))_{t \geq 0}$.
	Thus, 
	$\Sigma_{\mathrm{lin}}(A,B)$ is 
	polynomially ISS by Lemma~\ref{lem:poly_ISS_admissible}.
	\qed
\end{proof}

From an argument similar to that in Example~18 of \cite{Paunonen2014OM},
we see that if $\beta < \alpha$, then the condition
$\ran(B) \subset D((-A)^{\beta})$
may not lead to UGS.
\begin{example}
	\label{ex:diagonal}
	{\em
		Let  $A$ be
		the generator of a polynomially
		stable semigroup $(T(t))_{t \geq 0}$  with parameter $\alpha >0$
		on a Banach space $X$.
		Set $U := X$ and $B := (-A)^{-\beta}$ with $0< \beta < \alpha$.
		Taking the input $u(t) := T(t)y_0$ with $y_0\in X$, we obtain
		\[
		\left\|
		\int^t_0 T(t-s)Bu(s) \mathrm{d}s
		\right\| = \|t T(t) (-A)^{-\beta} y_0\|
		\] 
		for all $t \geq 0$.
		If the linear system $\Sigma_{\mathrm{lin}}(A,B)$ is UGS
		for $\mathcal{U} = L^{\infty}(\mathbb{R}_+,U)$, then
		the uniform boundedness principle
		implies that 
		\begin{equation}
		\label{eq:input_term_ex}
		\sup_{t \geq 0} \|t T(t) (-A)^{-\beta}\| < \infty.
		\end{equation}
		However, one can easily find polynomially stable semigroups with
		parameter $\alpha$
		for which the condition \eqref{eq:input_term_ex} does not hold.
		Hence, the condition $\ran(B) \subset D((-A)^{\beta})$ with $\beta <\alpha$ 
		does not imply 
		UGS in general.
		The case $\alpha = \beta$ remains open except in the diagonalizable 
		case studied in the next subsection. 
	}
\end{example}

\subsection{Polynomial input-to-state stability for
	diagonalizable linear systems}
In this subsection, we consider linear systems with
diagonalizable generators and finite-rank input operators.
We aim to refine the condition on the range of the
input operator  obtained in Proposition~\ref{prop:polyiss}.
To this end, we first review the definition and basic properties 
of diagonalizable operators; 
see Section~2.6 of \cite{Tucsnak2009} for details.
\begin{definition}
	\label{def:diagonalizable}
	{\em 
		Let $X$ be a Hilbert space.
		The linear operator $A :D(A) \subset X \to X$ is {\em diagonalizable}
		if $\varrho(A) \not= \emptyset$ and there exists a 
		Riesz basis $(\varphi_n)_{n \in \mathbb{N}}$ in $X$
		consisting of eigenvectors of $A$.
	}
\end{definition}

Throughout this subsection, we place the following assumption.
\begin{assumption}
	\label{assump:one_d_case}
	{\em
	Let $X$ be a Hilbert space with inner product $\langle \cdot, 
	\cdot \rangle$. The 
	operator
	$A:D(A) \subset X \to X $ is diagonalizable, and
	$(\varphi_n)_{n \in \mathbb{N}}$ is a Riesz basis in $X$ consisting
	of eigenvectors of $A$.
	The biorthogonal sequence for $(\varphi_n)_{n \in \mathbb{N}}$
	and the eigenvalue corresponding to the eigenvector $\varphi_n$
	are given by 
	$(\psi_n)_{n \in \mathbb{N}}$
	and $\lambda_n$, respectively.
}
\end{assumption}

\begin{proposition}
	\label{prop:diagonalizable}
	Suppose that Assumption~\ref{assump:one_d_case} is satisfied. 
	Then the following statements hold:
	\begin{enumerate}
		\renewcommand{\labelenumi}{{\em \arabic{enumi}.}}
		\item The operator $A$ may be written as
		\[
		Ax = \sum_{n=1}^{\infty} \lambda_n \langle 
		x , \psi_n
		\rangle \varphi_n\qquad \forall x \in D(A)
		\]
		and
		\[
		D(A) =
		\left\{
		x \in X :
		\sum_{n=1}^{\infty} |\lambda_n|^2 ~\!|\langle 
		x , \psi_n
		\rangle |^2 < \infty
		\right\}.
		\]
		\item 
		The operator $A$ is the generator of a 
		$C_0$-semigroup $(T(t))_{t\geq 0}$ on $X$ if and only if
		\[
		\sup_{n \in \mathbb{N}} \re \lambda_n < \infty.
		\]
		In this case,
		the exponential growth bound of 
		$(T(t))_{t\geq 0}$ is given by $\sup_{n \in \mathbb{N}} \re \lambda_n$,
		and for all $x \in X$ and $t \geq 0$,
		\[
		T(t)x = \sum_{n=1}^{\infty} e^{t \lambda_n } \langle 
		x , \psi_n
		\rangle \varphi_n.
		\]
	\end{enumerate}
\end{proposition}

Suppose that 
the eigenvalues $(\lambda_n)_{n \in \mathbb{N}}$ of
a diagonalizable operator $A$ satisfy $\re \lambda_n \leq 0$ for all 
$n \in \mathbb{N}$. 
Then
$A$ generates a uniformly bounded semigroup.
Moreover, $-A$ is sectorial in the sense of
\cite[Chapter~2]{Haase2006}, and hence 
the fractional power $(-A)^{\alpha}$
is well defined for every $\alpha > 0$. 
The domain of the fractional power $(-A)^{\alpha}$
is given by
\[
D((-A)^{\alpha}) =
\left\{
x \in X: \sum_{n=1}^{\infty} |\lambda_n|^{2\alpha}~\! |\langle x,\psi_n \rangle |^2
< \infty
\right\}
\] 
for all $\alpha > 0$, where $(\psi_n)_{n \in \mathbb{N}}$ is as in 
Assumption~\ref{assump:one_d_case}.

A diagonalizable operator is similar to a normal operator. Hence,
by Proposition~\ref{prop:spectral_prop}, 
a diagonalizable operator with
eigenvalues $(\lambda_n)_{n \in \mathbb{N}}$ 
generates a polynomially stable semigroup with parameter 
$\alpha>0$
if and only if
$\re \lambda_n <0$ for all $n \in \mathbb{N}$ and
there exist $C,p>0$ such that 
\begin{equation}
\label{eq:eigenvalue_geometric_cond}
|\im \lambda_n| \geq \frac{C }{ |\re \lambda_n|^{1/\alpha}} \qquad 
\text{if $\re \lambda_n > -p$}.
\end{equation}

We obtain a refined sufficient condition for linear systems with diagonalizable generators
and finite-rank input operators to be polynomially ISS. 
\begin{theorem}
	\label{thm:diagonalizable_polyISS}
	Let Assumption~\ref{assump:one_d_case} be satisfied and let
	$U$ be a Banach space.
	Suppose that the diagonalizable operator 
	$A$ generates a polynomially stable semigroup with parameter 
	$\alpha>0$ on $X$.
	If $B \in \mathcal{L}(U,X)$ is a finite-rank operator
	and satisfies
	$
	\ran(B) \subset 
	D((-A)^{\alpha})
	$,
	then $\Sigma_{\mathrm{lin}}(A,B)$  is polynomially ISS
	with parameter $\alpha$ for $\mathcal{U}= L^{\infty}(\mathbb{R}_+,U)$.
\end{theorem}
\begin{proof}
	By Lemma~\ref{lem:poly_ISS_admissible}, it suffices to show that 
	$B$ is infinite-time $L^{\infty}$-admissible for $(T(t))_{t\geq 0}$
	in the case $B \not=0$.
	
	By a property of a Riesz basis (see, e.g., 
	Proposition~2.5.2 of \cite{Tucsnak2009}), there exists a constant $M_1>0$
	such that
	\begin{equation}
	\label{eq:int_conv_bound}
	\left\|
	\int^t_0 
	T(s) B u(s)\mathrm{d}s 
	\right\|^2 \leq 
	M_1 \sum_{n=1}^{\infty}
	\left|
	\left\langle
	\int^t_0 T(s) Bu(s)\mathrm{d}s, \psi_n
	\right\rangle 
	\right|^2
	\end{equation}
	for all $u \in L^{\infty}(\mathbb{R}_+,U)$ and  $t\geq 0$.
	Since $B$ is a finite-rank operator,
	there is an orthonormal basis $(\xi_k)_{k=1}^m$ of
	the finite-dimensional space $\ran (B)$, where 
	$m\in\mathbb{N}$ is the dimension of $\ran (B)$. Then we obtain
	\[
	Bv = \sum_{k=1}^m \langle Bv, \xi_k \rangle \xi_k\qquad \forall v \in U.
	\]
	Since 
	\[
	\langle  T(s) \xi_k, \psi_n \rangle =
	\langle  \xi_k, T(s)^*\psi_n \rangle =
	e^{s \lambda_n} \langle  \xi_k, \psi_n \rangle
	\]
	for all $s \geq 0$, 
	it follows that
	\begin{align*}
		\left\langle
		\int^t_0 T(s) Bu(s)\mathrm{d}s, \psi_n
		\right\rangle  &=
		\left\langle
		\int^t_0  \sum_{k=1}^m \langle Bu(s), \xi_k \rangle T(s) \xi_k \mathrm{d}s, \psi_n
		\right\rangle \\
		&=
		\sum_{k=1}^m 
		\int^t_0  e^{s\lambda_n} 
		\langle Bu(s), \xi_k \rangle
		\mathrm{d}s ~\! \langle  \xi_k , \psi_n \rangle
	\end{align*}
	for all $u \in L^{\infty}(\mathbb{R}_+,U)$ and $t\geq 0$.
	Therefore,
	\begin{align*}
		\left|
		\left\langle
		\int^t_0 T(s) Bu(s)\mathrm{d}s, \psi_n
		\right\rangle 
		\right| &\leq\sum_{k=1}^m
		\int^t_0  e^{ s \re \lambda_n} 
		|\langle Bu(s), \xi_k \rangle  |\mathrm{d}s ~\! |\langle \xi_k, \psi_n \rangle|  \\
		&\leq \frac{\|B\| ~\! \|u\|_{\infty}}{|\re  \lambda_n |}
		\sum_{k=1}^m |\langle \xi_k, \psi_n \rangle|.
	\end{align*}

	Combining 
	$\xi_k \in  D((-A)^{\alpha}) $ with
	the geometric condition \eqref{eq:eigenvalue_geometric_cond} on 
	$(\lambda_n)_{n \in \mathbb{N}}$, we obtain
	\[
	\sum_{n=1}^{\infty} 
	\frac{\left|
		\langle \xi_k,\psi_n \rangle 
		\right|^2}{|\re \lambda_n|^2 } 
	\leq
	\frac{1}{p^2}
	\sum_{n=1}^{\infty} 
	\left|
	\langle \xi_k,\psi_n \rangle 
	\right|^2
	+
	\frac{1}{C^{2\alpha}}
	\sum_{n=1}^{\infty} 
	|\lambda_n|^{2\alpha}~\! |
	\langle \xi_k,\psi_n \rangle 
	|^2 =: c_k < \infty
	\]
	for every $k=1,\dots,m$. Therefore,
	\begin{align}
		\sum_{n=1}^{\infty}
		\left|
		\left\langle
		\int^t_0 T(s) Bu(s)\mathrm{d}s, \psi_n
		\right\rangle 
		\right|^2 &\leq
		m(\|B\| ~\!\|u\|_{\infty})^2
		\sum_{k=1}^m  \sum_{n=1}^{\infty}
		\frac{|\langle \xi_k, \psi_n \rangle|^2 }{|\re \lambda_n|^2} \notag \\
		&\leq m(\|B\| ~\!\|u\|_{\infty})^2 \sum_{k=1}^{m} c_k
		\label{eq:int_conv_sum}
	\end{align}
	for all $u \in L^{\infty}(\mathbb{R}_+,U)$ and $t\geq 0$.
	From the estimates \eqref{eq:int_conv_bound} and \eqref{eq:int_conv_sum},
	we obtain
	\begin{align*}
		\left\|
		\int^t_0 T(s) Bu(s) \mathrm{d}s
		\right\| \leq 
		\left(\|B\| ~\! \sqrt{mM_1\sum_{k=1}^{m} c_k } \right)\|u\|_{\infty}
	\end{align*}
	for all $u \in L^{\infty}(\mathbb{R}_+,U)$ and $t\geq 0$.
	Thus,
	$B$ is infinite-time $L^{\infty}$-admissible for $(T(t))_{t\geq 0}$.
	\qed
\end{proof}

We apply Theorem~\ref{thm:diagonalizable_polyISS} to
an Euler-Bernoulli beam with weak damping.
\begin{example}
	{\em 
		Consider
		a simply supported Euler-Bernoulli beam with weak damping,
		which is described by the following partial
		differential equation on $(0,1)$:
		\begin{align}
		\label{eq:EB_eq}
		\begin{cases}
			\dfrac{\partial^2 z}{\partial t^2} (\zeta,t) +
			\dfrac{\partial^4 z}{\partial \zeta^4} (\zeta,t) + 
			h(\zeta) \displaystyle
			\int^1_0 h(r) \dfrac{\partial z}{\partial t} (r,t) \mathrm{d}r +b(\zeta)u(t)= 0,\\
\hspace{190pt} 0<\zeta <1,~t\geq 0\vspace{5pt}\\ 
			z(0,t) = 0 = z(1,t),\quad \dfrac{\partial^2 z}{\partial \zeta^2} (0,t)
			=
			0
			= \dfrac{\partial^2 z}{\partial \zeta^2} (1,t),\quad t \geq 0 \vspace{5pt}\\
		z(\zeta,0) = z_0(\zeta),\quad 
			\dfrac{\partial z}{\partial t} (\zeta,0) = z_1(\zeta), \quad 0<\zeta <1,
			\end{cases}
		\end{align}
		where $b \in L^2(0,1)$ is a ``shaping function'' for the external input $u$ and 
		$h$ is the damping coefficient. Here we set
		$h(\zeta) := 1-\zeta$ for $\zeta \in (0,1)$.
		
		It is well known that the partial differential equation \eqref{eq:EB_eq}
		can be written as a first-order linear system in the following way;
		see, e.g., Exercise~3.18 of \cite{Curtain2020}.
		Define $X_0 := L^2(0,1)$ and 
		\[
		A_0 f := \frac{\mathrm{d}^4 f}{\mathrm{d}\zeta^4}
		\]
		with domain
		\[
		D(A_0) := \left\{
		f \in W^{4,2}(0,1):f(0) = 0=f(1) \text{~and~}  \frac{\mathrm{d}^2 f}{\mathrm{d}\zeta^2}(0) = 0 = \frac{\mathrm{d}^2 f}{\mathrm{d}\zeta^2}(1)
		\right\}.
		\]
		The operator $A_0$ has a positive self-adjoint square root $A_0^{1/2}=
		-\frac{\mathrm{d}^2}{\mathrm{d}\zeta^2}$
		with domain
		\[
		D\big(A_0^{1/2}\big) = \{
		f \in W^{2,2}(0,1):f(0) = 0=f(1)
		\}.
		\]
		The space $X := D(A_0^{1/2}) \times L^2(0,1)$ equipped with an
		inner product
		\[
		\left\langle
		\begin{bmatrix}
		x_1 \\ x_2
		\end{bmatrix},
		\begin{bmatrix}
		y_1 \\ y_2
		\end{bmatrix}
		\right\rangle 
		:= 
		\big\langle
		A_0^{1/2}x_1,A_0^{1/2}y_1
		\big\rangle_{L^2} + 
		\langle
		x_2,y_2
		\rangle_{L^2}
		\]
		is a Hilbert space.
		Define the operators $A_1:D(A_1) \subset X \to X$ and $B,H \in \mathcal{L}(\mathbb{C},X)$ by
		\[
		A_1 := 
		\begin{bmatrix}
		0 & I \\-A_0 &0
		\end{bmatrix}
		\]
		with domain $D(A_1) = D(A_0) \times D(A_0^{1/2})$ and
		\[
		Bv := 
		\begin{bmatrix} 
		0 \\ -bv
		\end{bmatrix},\quad
		Hv := 
		\begin{bmatrix} 
		0 \\ hv
		\end{bmatrix},
		\quad v \in \mathbb{C}.
		\]
		For
		\[
		x := 
		\begin{bmatrix}
		z \vspace{3pt}\\ \dfrac{\partial z}{\partial t}\vspace{4pt}
		\end{bmatrix},\quad
		x_0 :=
		\begin{bmatrix}
		z_0 \\ z_1
		\end{bmatrix},
		\]
		the partial differential equation \eqref{eq:EB_eq} can be written as
		\[
		\dot x(t) = (A_1-HH^*)x(t) + Bu(t),\quad t \geq 0;\qquad x(0) = x_0.
		\]
		
		The operator $A_1$ is diagonalizable with simple
		eigenvalues
		\[
		\lambda_n := i n^2 \pi^2,\quad \lambda_{-n} := -in^2\pi^2,\quad n\in\mathbb{N}.
		\]
		Since $A_1$ has compact resolvents
		by Lemma~3.2.12 of \cite{Curtain2020}, 
		it follows from
		Theorem~1 of \cite{Xu1996} that 
		$A := A_1-HH^*$ is also diagonalizable.
		Moreover,
		$A$ generates a polynomially stable semigroup
		with parameter $\alpha = 1$ by Corollary~6.6 of \cite{Chill2019}.
		Thus, Theorem~\ref{thm:diagonalizable_polyISS} shows that
		$\Sigma_{\text{lin}}(A, B)$ is polynomially ISS with parameter $\alpha = 1$
		if $\ran (B) \subset D(A) = D(A_1)$, i.e.,
		 $b \in D(A_0^{1/2})$.
	}
\end{example}

\subsection{Case where  eigenvalues near the imaginary axis have
uniformly separated imaginary parts}
We investigate 
how sharp the condition $\ran(B) \subset D((-A)^{\alpha})$ is.  
To this end, we employ the relation between Laplace-Carleson embeddings
and
infinite-time $L^{\infty}$-admissibility established in \cite{Jacob2021_arXiv}.

Let $A:D(A) \subset X \to X$ 
be diagonalizable  and generate a
strongly stable
semigroup $(T(t))_{t\geq 0}$ on $X$. 
Let
$B \in \mathcal{L}(\mathbb{C},X)$ be
represented as $Bv = bv$ for some $b \in X$ and all $v \in \mathbb{C}$.
Define the Borel measure $\nu$ on 
the open right half-plane $\{
\lambda \in \mathbb{C}: \re \lambda >0
\}$ by
\begin{align*}
\nu &:= \sum_{n \in \mathbb{N}} |\langle b,\psi_n\rangle|^2 \delta_{-\lambda_n},
\end{align*}
where 
$(\lambda_n)_{n \in \mathbb{N}}$ and 
	$(\psi_n)_{n \in \mathbb{N}}$ are as in Assumption~\ref{assump:one_d_case}
	and
$\delta_{-\lambda_n}$ is the Dirac measure at the point $-\lambda_n$
for $n \in \mathbb{N}$.
Define the Carleson square $Q_I$ 
associated to an interval $I \subset i \mathbb{R}$
and the dyadic stripe $S_k$ for $k \in \mathbb{Z}$  by
\begin{align}
Q_I &:= \{
\lambda \in \mathbb{C} : i \im \lambda \in I,~0 < \re \lambda < |I|
\}  \\
S_k &:= \{
\lambda \in \mathbb{C}  : 2^k \leq \re \lambda  < 2^{k+1}
\}.\label{eq:mu_QS_def}
\end{align}
Then Theorem~2.5 of \cite{Jacob2021_arXiv} shows that $B$ is 
infinite-time $L^{\infty}$-admissible for  $(T(t))_{t\geq 0}$ 
if and only if
\begin{equation}
\label{eq:Carleson}
\sum_{k \in \mathbb{Z}} \sup_{\substack{I\subset i \mathbb{R} \\ 
		\text{$I$ interval}} }
\frac{\nu(Q_I \cap S_k)}{|I|^2} < \infty.
\end{equation}

Using this equivalence of admissibility, we obtain
a necessary and sufficient condition for polynomial ISS.
We
write
\begin{equation}
\label{eq:Phi_k_def}
\Phi_k := 
\begin{cases}\displaystyle
\sup_{-\lambda_n \in S_k} 
\frac{|\langle b,\psi_n\rangle|^2}{|\re \lambda_n|^2} & \text{if $\{-\lambda_n:n \in \mathbb{N}\}
	\cap S_k \not= \emptyset$} \vspace{3pt}\\
0 & \text{if $\{-\lambda_n:n \in \mathbb{N}\}
	\cap S_k = \emptyset$}
\end{cases}
\end{equation}
for $k \in \mathbb{Z}$.
\begin{theorem}
	Let Assumption~\ref{assump:one_d_case} hold and let $B \in \mathcal{L}(\mathbb{C},X)$
	be represented as $Bv = bv$ for some $b \in X$ and all $v \in \mathbb{C}$.
	Assume
	that there exists $p >0$ such that 
	\begin{equation}
	\label{eq:delta_12}
	\inf\{|\im \lambda_n - \im \lambda_{m}| :
	n,m \in \mathbb{N},~n\not=m,~|\re \lambda_n|,|\re\lambda_m| < 
p
	 \} > 0.
	\end{equation}
	Then $\Sigma_{\mathrm{lin}}(A,B)$ is polynomially ISS with parameter $\alpha>0$ for $\mathcal{U} = L^{\infty}(\mathbb{R}_+,U)$
	if and only if
	 the diagonalizable operator 
	$A$ generates a polynomially stable semigroup with parameter 
	$\alpha$ and 
		\begin{equation}
		\label{eq:Carleson2}
		\sum_{k \in \mathbb{Z}} \Phi_k < \infty,
		\end{equation}
		where 
		$\Phi_k$ is defined as in \eqref{eq:Phi_k_def} for $k \in \mathbb{Z}$.
\end{theorem}
\begin{proof}
	We first note that a polynomially stable semigroup is
	strongly stable.
	By Lemma~\ref{lem:poly_ISS_admissible},
	it suffices to show that the conditions \eqref{eq:Carleson} and
	\eqref{eq:Carleson2} are equivalent under the assumption \eqref{eq:delta_12}.

	Let $p>0$ satisfy \eqref{eq:delta_12}. There exists 
	$d>0$ such that $|\im \lambda_n - \im \lambda_{m}| \geq d$
	for all $n,m \in \mathbb{N}$ satisfying 
	$n\not=m$ and $|\re \lambda_n|, |\re\lambda_m| < p$.
	If an interval $I \subset i \mathbb{R}$ satisfies 
	$|I| \geq  d$,
	then for all $k \in \mathbb{Z}$,
	\[
	\frac{\nu(Q_I \cap S_k)}{|I|^2} \leq \frac{\nu(Q_I \cap S_k)}{d^2} \leq 
	\frac{\nu(S_k)}{d^2}.
	\]
	
	Suppose next that an interval $I \subset i \mathbb{R}$ satisfies 
	$|I| <  d$.
	If $k \in \mathbb{Z}$ satisfies
	$p \leq 2^{k+1}$,
	then $\nu(Q_I \cap S_k) >0$ implies $|I| \geq p/2$, and
	therefore
	\[
	\frac{\nu(Q_I \cap S_k)}{|I|^2} \leq \frac{4 \nu(Q_I \cap S_k)}{p^2} \leq 
	\frac{4 \nu(S_k)}{p^2}.
	\]
	Let  $k \in \mathbb{Z}$ satisfy 
	$p> 2^{k+1}$. 
	For $-\lambda_n,-\lambda_m \in S_k$ with $n\not=m$, we obtain
	$|\im \lambda_n - \im \lambda_{m}| \geq d$ by assumption.
	Recalling that the interval $I$ is chosen so that
	$|I| < d$, we have that $Q_I \cap S_k$ contains
	at most one element of $(-\lambda_n)_{n \in \mathbb{N}}$. Since
	$\nu(Q_I \cap S_k) = |\langle b,\psi_n\rangle|^2$ for some $n \in \mathbb{N}$ with
	$-\lambda_n\in S_k$  or $\nu(Q_I \cap S_k) = 0$,
	it follows that
	\[
	\frac{\nu(Q_I \cap S_k)}{|I|^2} \leq 
	\Phi_k.
	\]

	We have shown that for every interval $I \subset i \mathbb{R}$
	and $k \in \mathbb{Z}$,
	\[
	\frac{\nu(Q_I \cap S_k)}{|I|^2} \leq 
	\max\left\{
	\frac{\nu(S_k)}{d^2},~
	\frac{4 \nu(S_k)}{p^2},~
	\Phi_k
	\right\}.
	\]
	Hence
	\begin{align*}
		\sum_{k \in \mathbb{Z}} \sup_{\substack{I\subset i \mathbb{R} \\ 
				\text{$I$ interval}} }
		\frac{\nu(Q_I \cap S_k)}{|I|^2} &\leq 
		\left(
		\frac{1}{d^2} + \frac{4}{p^2}
		\right)
		\sum_{k \in \mathbb{Z}} \nu(S_k) + 
		\sum_{k \in \mathbb{Z}} \Phi_k\\
		&=
		\left(
		\frac{1}{d^2} + \frac{4 }{p^2}
		\right)
		\sum_{n \in \mathbb{N}} |\langle b,\psi_n\rangle|^2 + 
		\sum_{k \in \mathbb{Z}} \Phi_k .
	\end{align*}
	Since $b \in X$, it follows that $\sum_{n \in \mathbb{N}} |\langle b,\psi_n\rangle|^2 < \infty$.
	Therefore, 
	\eqref{eq:Carleson2} implies 
	\eqref{eq:Carleson}.
	
	Conversely, 
	for all $k \in \mathbb{Z}$,
	if $-\lambda_n \in S_k$, then
	\[
	\frac{|\langle b,\psi_n\rangle|^2}{|\re \lambda_n|^2} \leq \sup_{\substack{I\subset i \mathbb{R} \\ 
			\text{$I$ interval}} }
	\frac{\nu(Q_I \cap S_k)}{|I|^2},
	\]
	and hence
	\[
	\Phi_k\leq \sup_{\substack{I\subset i \mathbb{R} \\ 
			\text{$I$ interval}} }
	\frac{\nu(Q_I \cap S_k)}{|I|^2}.
	\]
	This yields
	\[
	\sum_{k \in \mathbb{Z}}
	\Phi_k \leq 
	\sum_{k \in \mathbb{Z}}\sup_{\substack{I\subset i \mathbb{R} \\ 
			\text{$I$ interval}} }
	\frac{\nu(Q_I \cap S_k)}{|I|^2}.
	\]
	Thus, \eqref{eq:Carleson} implies 
	\eqref{eq:Carleson2}.\qed
\end{proof}

For $k \in \mathbb{Z}$,
define 
\begin{equation*}
{\widetilde \Phi}_k := 
\begin{cases}\displaystyle
\sup_{-\lambda_n \in S_k} 
|\lambda_n|^{2\alpha} ~\! |\langle b,\psi_n\rangle|^2  & \text{if $\{-\lambda_n:n \in \mathbb{N}\}
	\cap S_k \not= \emptyset$} \vspace{3pt}\\
0 & \text{if $\{-\lambda_n:n \in \mathbb{N}\}
	\cap S_k = \emptyset$}.
\end{cases}
\end{equation*}
A routine calculation shows that 
if $\lim_{n \to \infty}\re \lambda_n = 0$ and if there exists $C>0$ such that 
\[
|\im \lambda_n| - \frac{C}{|\re \lambda_n|^{1/\alpha}} \to 0\qquad \text{as $n \to \infty$},
\]
then the condition  \eqref{eq:Carleson2} is equivalent to
\[
\sum_{k \in \mathbb{Z}} {\widetilde \Phi}_k < \infty.
\]
From this, we observe that 
the condition \eqref{eq:Carleson2} is milder than $b \subset D((-A)^{\alpha})$.
However, 
the following example shows that
if the assumption \eqref{eq:delta_12} is not satisfied,
then $b \in D((-A)^{\alpha})$ may be necessary and sufficient for 
infinite-time $L^{\infty}$-admissibility. 
\begin{example}
	{\em
	Consider a diagonalizable operator $A$ whose
	eigenvalues $(\lambda_n)_{n \in \mathbb{N}}$  are given by
	\[
	\lambda_n := -\frac{1}{2^k} - i2^k, \quad 
	2^k \leq n \leq 2^{k+1} - 1,~ k \in \mathbb{N}_0.
	\]
	Since $(\lambda_n)_{n \in \mathbb{N}}$ satisfies 
	the geometric condition
	\eqref{eq:eigenvalue_geometric_cond} with $\alpha = 1$, 
	it follows that $A$ generates a polynomially stable semigroups
	with parameter $\alpha = 1$.
	For all $k \in \mathbb{N}_0$, taking intervals 
	$I \subset i \mathbb{R}$
	with center $i2^k$, we obtain
	\begin{align*}
		\sup_{\substack{I\subset i \mathbb{R} \\ 
				\text{$I$ interval}} }
		\frac{\nu(Q_I \cap S_{-k})}{|I|^2} &=
		\frac{\sum_{n=2^k}^{2^{k+1}-1} |\langle b,\psi_n\rangle|^2}{1/2^{2k}} \\
		&\geq 
		\frac{1}{2} \sum_{n=2^k}^{2^{k+1}-1} | \lambda_n|^{2} ~\! |\langle b,\psi_n\rangle|^2.
	\end{align*}
	This yields
	\begin{align*}
	\sum_{n=1}^{\infty} | \lambda_n|^{2}~\!  |\langle b,\psi_n\rangle|^2 &= 
	\sum_{k \in \mathbb{N}_0}\sum_{n=2^k}^{2^{k+1}-1} | \lambda_n|^{2}~\!  |\langle b,\psi_n\rangle|^2\\
	&\leq 2 \sum_{k \in \mathbb{Z}}\sup_{\substack{I\subset i \mathbb{R} \\ 
			\text{$I$ interval}} }
	\frac{\nu(Q_I \cap S_k)}{|I|^2}.
	\end{align*}
	Thus, infinite-time $L^{\infty}$-admissibility 
	implies $b \in D(A)$.
}
\end{example}

\section{Polynomial integral input-to-state stability of bilinear systems}
\label{sec:PolyiISS}
In the previous section, we saw that  polynomial ISS
is restrictive even for linear systems with bounded input operators.
This is because infinite-time $L^{\infty}$-admissibility cannot be achieved
for all bounded input operators due to the weak asymptotic 
property of polynomially stable
semigroups.
This motivates us to study a semi-uniform version of 
integral input-to-state stability, which provides norm estimates of
trajectories with respect to a kind of energy fed into systems.

We recall a stability notion for systems without inputs; see
\cite[Definition~5]{Mironchenko2018}.
\begin{definition}
	{\em
		The semi-linear system $\Sigma(A,F)$ is called 
		{\em uniformly globally stable at zero} if
			the following two conditions hold:
		\begin{enumerate}
			\item
			$\Sigma(A,F)$ is forward complete.
			\item 
			There exists $\gamma \in \mathcal{K}_{\infty}$
			such that 
			\[
			\|\phi(t,x_0,0)\| \leq \gamma(\|x_0\|)
			\]
			 for all $x_0 \in X$ and $t \geq 0$.
		\end{enumerate}
}
\end{definition}	
	
We define the concept of semi-uniform integral input-to-state
stability.
\begin{definition}
	{\em
	The semi-linear system $\Sigma(A,F)$ is called 
	{\em semi-uniformly integral 
		input-to-state stable (semi-uniformly iISS)} if 
	the following two conditions hold:
	\begin{enumerate}
		\item
		$\Sigma(A,F)$ is uniformly globally stable at zero.
		\item
		There exist $\kappa \in \mathcal{KL}$, $\theta \in \mathcal{K}_{\infty}$, and
		$\mu \in \mathcal{K}$ such that
		\begin{equation}
		\label{eq:semi_iISS}
		\|\phi(t,x_0,u)\| \leq 
		\kappa (\|x_0\|_A, t) + \theta \left(
		\int^t_0 \mu (\|u(s)\|_U)\mathrm{d}s
		\right)
		\end{equation}
		for all $x_0 \in D(A)$, $u \in \mathcal{U}$, and $t \geq 0$.
	\end{enumerate}
	In particular, if there exists $\alpha >0$ such that for
	all $r > 0$,
	$\kappa(r,t) = O(t^{-1/\alpha})$ as $t \to \infty$,  then $\Sigma(A,F)$ is called 
	{\em polynomially integral input-to-state stable (polynomially iISS) 
		with parameter $\alpha >0$}.
	}
\end{definition}

Note that the integral $\int^t_0 \mu (\|u(s)\|_U)\mathrm{d}s$ in the right-hand side of
the inequality
\eqref{eq:semi_iISS} may be infinite. In that case, the inequality 
\eqref{eq:semi_iISS}  trivially holds.

For every generator $A$ of 
a semi-uniformly stable semigroup and every
bounded input operator $B$,
the linear system
$\Sigma_{\mathrm{lin}}(A,B)$ is semi-uniform iISS.
Moreover, if the linear system
$\Sigma_{\mathrm{lin}}(A,B)$ is semi-uniform iISS, then
$\Sigma_{\mathrm{lin}}(A,B)$ is strong iISS in the sense of
Definition~4 in \cite{Nabiullin2018}.
This can be seen by 
using the equality $\phi(t,x_0,u) = \phi(t,x_0,0) +\phi(t,0,u)$
as in 
the case of semi-uniform ISS discussed in Theorem~\ref{thm:semi_to_strong}.

The aim of this section is to give
a sufficient condition for
bilinear systems satisfying Assumption~\ref{assump:Lipschitz}
to be polynomially iISS for $\mathcal{U} = L^{\infty}(\mathbb{R}_+,U)$.
We prove that 
if the nonlinear operator  additionally satisfies a certain
smoothness assumption, then the bilinear system is
polynomially iISS.
To this end, we use a non-Lyapunov method devised in Theorem~4.2 of \cite{Mironchenko2016} for
uniform iISS.
\begin{theorem}
	Let $A$ be the generator of a polynomially stable semigroup
	$(T(t))_{t\geq 0}$  with parameter $\alpha >0$ on a Banach space $X$.
	Suppose that the nonlinear operator $F$ satisfies
	Assumption~\ref{assump:Lipschitz} for another Banach space $U$ 
	and that
	there exist $K>0$ and $\chi \in \mathcal{K}$ such that
	for all $\xi \in X$, $v \in U$, and $t \geq 0$, 
	\[
	\|T(t)G(\xi,v) \| \leq \frac{K \|\xi\| \chi(\|v\|_U)}{(t+1)^{1/\alpha}}.
	\]
	Then the bilinear system $\Sigma(A,F)$ is
	polynomially iISS with parameter $\alpha$ for $\mathcal{U} = L^{\infty}(\mathbb{R}_+,U)$.
\end{theorem}

\begin{proof}
	Since $(T(t))_{t\geq 0}$ is polynomially stable  with parameter 
	$\alpha >0$, there exists $M \geq 1$ such that
	\[
	\|T(t)\| \leq M,\quad 
	\|T(t) R(1,A)\| \leq \frac{M}{(t+1)^{1/\alpha}}\qquad \forall t \geq 0.
	\] 	
	By Gronwall's inequality 
	(see Appendix A of \cite{Pata2011} for a simple proof), we have that for all $x_0 \in X$, $u \in L^{\infty}(\mathbb{R}_+,U)$, and $t \geq 0$,
	\begin{align*}
	\|x(t)\| &\leq M
	\left(
	\|x_0\| + \|B\| \int^t_0 \|u(s)\|_U \mathrm{d}s
	\right) + K 
	\int^t_0 \|x(s)\| \chi (\|u(s)\|_U)\mathrm{d}s \\
	&\leq M(\|x_0\| + t\|B\|~\!\|u\|_{\infty}) e^{tK \chi(\|u\|_{\infty}) }
	\end{align*}
	as long as $x$ is a mild solution of $\Sigma(A,F)$ on $[0,t]$.
	Hence $\Sigma(A,F)$ is forward complete by the remark following
	Assumption~\ref{assump:Lipschitz}.
	Moreover, 
	$F(\xi,0) = 0$ for all $\xi \in X$ under Assumption~\ref{assump:Lipschitz}. Therefore,
	if $u(t) \equiv 0$, then the mild solution $x$  of $\Sigma(A,F)$
	satisfies
	\[
	\|x(t)\| = \|T(t)x_0\| \leq M\|x_0\|
	\]
	for all $x_0 \in X$ and $t \geq 0$, which implies that $\Sigma(A,F)$ is uniformly
	globally stable at zero.
	
	Take $x_0 \in D(A)$ and $u \in L^{\infty}(\mathbb{R}_+,U)$.
	The mild solution $x$ of $\Sigma(A,F)$ satisfies
	\begin{align*}
	\|x(t)\| \leq \frac{M}{(t+1)^{1/\alpha}} \|x_0\|_A + 
	\int^t_0 
	\bigg( &M \|B\|~\!\|u(s)\|_U \\
	&+ 
	\frac{K}{(t-s+1)^{1/\alpha}} \|x(s)\| \chi (\|u(s)\|_U) \mathrm{d}s
	\bigg)
	\end{align*}
	for all $t \geq 0$.
	Define $z(t) := (t+1)^{1/\alpha} \|x(t)\|$ for $t \geq 0$.
	Then
	\begin{align*}
	z(t) \leq M 
	&\left(\|x_0\|_A +  \|B\| (t+1)^{1/\alpha} \int^t_0 \|u(s)\|_U \mathrm{d}s\right) \\
	&\qquad + 
	K \int^t_0 
	\left(
	\frac{t+1}{(t-s+1)(s+1)} 
	\right)^{1/\alpha}
	z(s)  \chi (\|u(s)\|_U) \mathrm{d}s
	\end{align*}
	for all $t \geq 0$.
	Since 
	\[
	\max_{0\leq s \leq t}
	\left(
	\frac{t+1}{(t-s+1)(s+1)} 
	\right)^{1/\alpha} = 1,
	\]
	Gronwall's inequality implies that for all $t \geq 0$,
	\[
	z(t) \leq 
	M \left(
	\|x_0\|_A  +\|B\|  (t+1)^{1/\alpha} \int^t_0 \|u(s)\|_U \mathrm{d}s 
	\right) e^{
		K \int^t_0  \chi (\|u(s)\|_U) \mathrm{d}s
	},
	\]
	which is equivalent to
	\[
	\|x(t)\| \leq 
	M \left(
	\frac{\|x_0\|_A  }{(t+1)^{1/\alpha} } +\|B\|\int^t_0 \|u(s)\|_U \mathrm{d}s 
	\right) e^{
		K \int^t_0  \chi (\|u(s)\|_U) \mathrm{d}s
	}.
	\] 
	Using the inequality 
	\[
	\ln (1+ae^b) \leq \ln(1+a)+b\qquad \forall a,b \geq 0,
	\]
	we obtain
	\begin{align*}
		&\ln (1+\|x(t)\|) \\
		&\leq 
		\ln \left(
		1+ M \left(
		\frac{\|x_0\|_A  }{(t+1)^{1/\alpha} } +\|B\|\int^t_0 \|u(s)\|_U \mathrm{d}s 
		\right)
		\right) +
		K \int^t_0  \chi (\|u(s)\|_U) \mathrm{d}s
	\end{align*}
	for all $t \geq 0$.
	Since 
	\[
	\ln(1+ a+b) \leq \ln (1+a) + \ln (1+b)\qquad \forall a,b \geq 0,
	\]
	it follows that 
	\begin{align*}
	\ln (1+\|x(t)\|)  \leq 
	\ln\left(
	1+ 
	\frac{M\|x_0\|_A  }{(t+1)^{1/\alpha} } 
	\right) 
	&+ 
	\ln\left(
	1
	+ 
	M\|B\|\int^t_0 \|u(s)\|_U \mathrm{d}s 
	\right)  \\
	&+
	K \int^t_0  \chi (\|u(s)\|_U) \mathrm{d}s
	\end{align*}
	for all $t \geq 0$.
	The inverse function of $q(r) := \ln (1+r)$, $r \geq 0$, is given by $q^{-1}(r) =  e^{r}-1$. Using the inequality
	\[
	e^{a+b}-1 \leq (e^{2a} -1) + (e^{2b} - 1)\qquad \forall a, b \geq 0
	\]
	twice, we obtain
	\begin{align*}
		\|x(t)\| \leq 
		\left(1+
		\frac{M\|x_0\|_A  }{(t+1)^{1/\alpha}} 
		\right)^2 - 1 
		&+
		\left(
		1+ 
		M\|B\|\int^t_0 \|u(s)\|_U \mathrm{d}s 
		\right)^4 - 1 \\
		&+
		e^{4K \int^t_0  \chi (\|u(s)\|_U) \mathrm{d}s} - 1
	\end{align*}
	for all $t \geq 0$.
	Thus, 	
	the bilinear system $\Sigma(A,F)$ is polynomially iISS 
	with parameter $\alpha$, where $\kappa \in \mathcal{KL}$, $\theta \in \mathcal{K}_{\infty}$, and
	$\mu \in \mathcal{K}$ 
	are given by
	\begin{align*}
		\kappa(r,t) &:= \left(
		\frac{Mr}{(t+1)^{1/\alpha}}
		\right)^2 +  
		\frac{2Mr}{(t+1)^{1/\alpha}} \\
		\theta(r) &:= r^4+4r^3+6r^2+4r + e^{r}-1\\
		\mu(r) &:= \max \{M\|B\| r, 4K \chi(r) \}
	\end{align*}
	for the estimate \eqref{eq:semi_iISS}.
	\qed
\end{proof}

\section{Conclusion}
We have introduced the notion of semi-uniform ISS
and have established its characterization based on attractivity properties.
We have given sufficient conditions for linear systems to be 
polynomially ISS.
In the sufficient conditions, the range of the input operator is restricted,
depending on 
the polynomial decay rate of the product of the $C_0$-semigroup and
the resolvent of its generator.
We have also shown that a class of
bilinear systems are polynomially iISS if the 
nonlinear operator satisfies a smoothness assumption
like the range condition of input operators for 
polynomial ISS of linear systems.
Important directions for future research are
to explore the relation between 
semi-uniform ISS and semi-uniform iISS and
to construct Lyapunov functions for
polynomial ISS and polynomial iISS.

\begin{acknowledgements}
	The author  would  like to thank the editor and anonymous reviewers
	for their careful reading of the manuscript and many insightful comments, 
	which, in particular, make the argument in Example~\ref{ex:diagonal}
	simpler.
\end{acknowledgements}


\begin{thebibliography}{10}
	\providecommand{\url}[1]{{#1}}
	\providecommand{\urlprefix}{URL }
	\expandafter\ifx\csname urlstyle\endcsname\relax
	\providecommand{\doi}[1]{DOI~\discretionary{}{}{}#1}\else
	\providecommand{\doi}{DOI~\discretionary{}{}{}\begingroup
		\urlstyle{rm}\Url}\fi
	
	\bibitem{Angeli1999}
	Angeli, D., Sontag, E.D.: {Forward completeness, unboundedness observability,
		and their Lyapunov characterizations}.
	\newblock Systems Control Lett. \textbf{38}, 209--217 (1999)
	
	\bibitem{Arendt2001}
	Arendt, W., Batty, C.J.K., Hieber, M., Neubrander, F.: Vector-valued Laplace
	Transforms and Cauchy Problems.
	\newblock Basel: Birkh\"auser (2001)
	
	\bibitem{Batkai2006}
	B\'atkai, A., Engel, K.J., Pr\"uss, J., Schnaubelt, R.: Polynomial stability of
	operator semigroups.
	\newblock Math. Nachr. \textbf{279}, 1425--1440 (2006)
	
	\bibitem{Batty2008}
	Batty, C.J.K., Duyckaerts, T.: {Non-uniform stability for bounded semi-groups
		on Banach spaces}.
	\newblock J. Evol. Equations \textbf{8}, 765--780 (2008)
	
	\bibitem{Borichev2010}
	Borichev, A., Tomilov, Y.: Optimal polynomial decay of functions and operator
	semigroups.
	\newblock Math. Ann. \textbf{347}, 455--478 (2010)
	
	\bibitem{Cazenave1998}
	Cazenave, T., Haraux, A.: An Introduction to Semilinear Evolution Equations.
	\newblock New York: Oxford Univ. Press (1998)
	
	\bibitem{Chill2019}
	Chill, R., Paunonen, L., Seifert, D., Stahn, R., Tomilov, Y.: Non-uniform
	stability of damped contraction semigroups (2019).
	\newblock \urlprefix\url{https://arxiv.org/pdf/1911.04804.pdf}.
	\newblock To appear in Anal. PDE
	
	\bibitem{Chill2020}
	Chill, R., Seifert, D., Tomilov, Y.: Semi-uniform stability of operator
	semigroups and energy decay of damped waves.
	\newblock Philos. Trans. Roy. Soc. A \textbf{378}, 20190614 (2020)
	
	\bibitem{Curtain2020}
	Curtain, R.F., Zwart, H.J.: An Introduction to Infinite-Dimensional Systems: A
	State Space Approach.
	\newblock New York: Springer (2020)
	
	\bibitem{Dashkovskiy2013}
	Dashkovskiy, S., Mironchenko, A.: Input-to-state stability of
	infinite-dimensional control systems.
	\newblock Math. Control Signals Systems \textbf{25}, 1--35 (2013)
	
	\bibitem{Engel2000}
	Engel, K.J., Nagel, R.: One-Parameter Semigroups for Linear Evolution
	Equations.
	\newblock New York: Springer (2000)
	
	\bibitem{Haase2006}
	Haase, M.: The Functional Calculus for Sectorial Operators.
	\newblock Basel: Birkh\"auser (2006)
	
	\bibitem{Hosfeld2022}
	Hosfeld, R., Jacob, B., Schwenninger, F.: Integral input-to-state stability of
	unbounded bilinear control systems.
	\newblock Math. Control Signals Systems  (2022).
	\newblock \doi{10.1007/s00498-021-00308-9}
	
	\bibitem{Jacob2020}
	Jacob, B., Mironchenko, A., Partington, J.R., Wirth, F.: {Noncoercive Lyapunov
		functions for input-to-state stability of infinite-dimensional systems}.
	\newblock SIAM J. Control Optim. \textbf{58}, 2952--2978 (2020)
	
	\bibitem{Jacob2018}
	Jacob, B., Nabiullin, R., Partington, J.R., Schwenninger, F.L.:
	{Infinite-dimensional input-to-state stability and Orlicz spaces}.
	\newblock SIAM J. Control Optim. \textbf{56}, 868--889 (2018)
	
	\bibitem{Jacob2021_arXiv}
	Jacob, B., Partington, J.R., Pott, S., Rydhe, E., Schwenninger, F.L.:
	{Laplace-Carleson embeddings and infinity-norm admissibility} (2021).
	\newblock \urlprefix\url{https://arxiv.org/pdf/2109.11465.pdf}
	
	\bibitem{Jayawardhana2008}
	Jayawardhana, B., Logemann, H., Ryan, E.P.: Infinite-dimensional feedback
	systems: the circle criterion and input-to-state stability.
	\newblock Commun. Inf. Systems \textbf{8}, 413--444 (2008)
	
	\bibitem{Karafyllis2016}
	Karafyllis, I., Krstic, M.: {ISS with respect to boundary disturbances for 1-D
		parabolic PDEs}.
	\newblock IEEE Trans. Automat. Control \textbf{61}, 3712--3724 (2016)
	
	\bibitem{Liu2005PDR}
	Liu, Z., Rao, B.: Characterization of polynomial decay rate for the solution of
	linear evolution equation.
	\newblock Angew. Math. Phys. \textbf{56}, 630--644 (2005)
	
	\bibitem{Mironchenko2016}
	Mironchenko, A., Ito, H.: Characterizations of integral input-to-state
	stability for bilinear systems in infinite dimensions.
	\newblock Math. Control Related Fields \textbf{6}, 447--466 (2016)
	
	\bibitem{Mironchenko2020}
	Mironchenko, A., Prieur, C.: Input-to-state stability of infinite-dimensional
	systems: recent results and open questions.
	\newblock SIAM Review \textbf{62}, 529--614 (2020)
	
	\bibitem{Mironchenko2018}
	Mironchenko, A., Wirth, F.: Characterizations of input-to-state stability for
	infinite-dimensional systems.
	\newblock IEEE Trans. Automat. Control \textbf{63}, 1692--1707 (2018)
	
	\bibitem{Nabiullin2018}
	Nabiullin, R., Schwenninger, F.L.: Strong input-to-state stability for
	infinite-dimensional linear systems.
	\newblock Math. Control Signals Systems \textbf{30}, Art. no. 4 (2018)
	
	\bibitem{Pata2011}
	Pata, V.: {Uniform estimates of Gronwall type}.
	\newblock J. Math. Anal. Appl. \textbf{373}, 264--270 (2011)
	
	\bibitem{Paunonen2011}
	Paunonen, L.: {Perturbation of strongly and polynomially stable Riesz-spectral
		operators}.
	\newblock Systems Control Lett. \textbf{60}, 234--248 (2011)
	
	\bibitem{Paunonen2012SS}
	Paunonen, L.: Robustness of strongly and polynomially stable semigroups.
	\newblock J. Funct. Anal. \textbf{263}, 2555--2583 (2012)
	
	\bibitem{Paunonen2013SS}
	Paunonen, L.: Robustness of polynomial stability with respect to unbounded
	perturbations.
	\newblock Systems Control Lett. \textbf{62}, 331--337 (2013)
	
	\bibitem{Paunonen2014OM}
	Paunonen, L.: Polynomial stability of semigroups generated by operator
	matrices.
	\newblock J. Evol. Equations \textbf{14}, 885--911 (2014)
	
	\bibitem{Rozendaal2019}
	Rozendaal, J., Seifert, D., Stahn, R.: {Optimal rates of decay for operator
		semigroups on Hilbert spaces}.
	\newblock Adv. Math. \textbf{346}, 359--388 (2019)
	
	\bibitem{Schmid2019}
	Schmid, J.: {Weak input-to-state stability: characterizations and
		counterexamples}.
	\newblock Math. Control Signals Systems \textbf{31}, 433--454 (2019)
	
	\bibitem{Sontag1989}
	Sontag, E.D.: Smooth stabilization implies coprime factorization.
	\newblock IEEE Trans. Automat. Control \textbf{34}, 435--443 (1989)
	
	\bibitem{Sontag1998}
	Sontag, E.D.: {Comments on integral variants of ISS}.
	\newblock Systems Control Lett. \textbf{34}, 93--100 (1998)
	
	\bibitem{Sontag1996}
	Sontag, E.D., Wang, Y.: New characterizations of input-to-state stability.
	\newblock IEEE Trans. Automat. Control \textbf{41}, 1283--1294 (1996)
	
	\bibitem{Su2020}
	Su, P., Tucsnak, M., Weiss, G.: Stabilizability properties of a linearized
	water waves system.
	\newblock Systems Control Lett. \textbf{138}, 104672 (2020)
	
	\bibitem{Tucsnak2009}
	Tucsnak, M., Weiss, G.: Observation and Control of Operator Semigroups.
	\newblock Basel: Birkh\"auser (2009)
	
	\bibitem{Weiss1989}
	Weiss, G.: Admissibility of unbounded control operators.
	\newblock SIAM J. Control Optim. \textbf{27}, 527--545 (1989)
	
	\bibitem{Xu1996}
	Xu, C.Z., Sallet, G.: {On spectrum and Riesz basis assignment of
		infinite-dimensional linear systems by bounded linear feedbacks}.
	\newblock SIAM J. Control Optim. \textbf{34}, 521--541 (1996)
	
\end{thebibliography}
\end{document}